\documentclass[a4paper,10pt]{amsart}

\usepackage[english]{babel}
\usepackage[utf8]{inputenc}
\usepackage[T1]{fontenc}

\usepackage[babel=true]{csquotes}

\usepackage{comment}

\usepackage{amsfonts}
\usepackage{amsthm} 
\usepackage{amsmath}
\usepackage{amssymb}

\usepackage{multirow} 
\usepackage{kbordermatrix}

\usepackage{brunnian}

\usetikzlibrary{arrows}

\theoremstyle{definition}
\newtheorem{defn}{Definition}[section] 

\theoremstyle{plain}
\newtheorem{lem}[defn]{Lemma} 
\newtheorem{prop}[defn]{Proposition}
\newtheorem{thm}[defn]{Theorem}
\newtheorem{cor}[defn]{Corollary}

\newtheorem*{divers}{Proposition \ref{prop groupe cable}}

\theoremstyle{remark}
\newtheorem{rem}[defn]{Remark}
\newtheorem{ex}[defn]{Example}

\newcommand{\C}{\mathbb{C}}
\newcommand{\R}{\mathbb{R}}

\newcommand{\Z}{\mathbb{Z}}
\newcommand{\N}{\mathbb{N}}
\newcommand{\F}{\mathbb{F}}

\title{The $L^2$-Alexander invariant detects the unknot}
\author{Fathi Ben Aribi}

\address{Institut de math\'ematiques de Jussieu–Paris Rive gauche, universit\'e Paris-Diderot (Paris-7), UFR de math\'ematiques, case 7012, b\^atiment Sophie-Germain, 75205 Paris
cedex 13, France}

\email{benaribi@math.jussieu.fr}
\begin{document}

\renewcommand{\proofname}{Proof}


\maketitle

\begin{abstract}
In this article, we present some of the properties of the $L^2$-Alexander invariant of a knot defined in \cite{LZ06a}, some of which are similar to those of the classical Alexander polynomial. Notably we prove that the $L^2$-Alexander invariant detects the trivial knot.

\end{abstract}

\renewcommand{\kbldelim}{(}
\renewcommand{\kbrdelim}{)}

\section{Introduction}

In 1923, Alexander introduced the first polynomial invariant of knots. It was nothing short of a revolution, since this invariant was easy to compute and powerful enough to distinguish most of the tabulated prime knots. However, the Alexander polynomial is not a complete invariant, not even among prime knots. In particular it does not detect the unknot.

In 1976, Atiyah laid the foundations of the theory of $L^2$-invariants. The idea is roughly the following: algebraic topology has many invariants that involve finite dimensional vector spaces and linear maps; by doing similar processes with infinite dimensional Hilbert spaces - like $\ell^2(G)$ where $G$ is a group - and operators on these spaces, we obtain the so-called $L^2$-invariants.

In the nineties, Carey-Mathai, Lott, Lück-Rothenberg,
and Novikov-Shubin developed the theory of $L^2$-torsions, an $L^2$-analog of the Reidemeister torsion theory.

Finally, in 2006, Li and Zhang introduced the $L^2$-Alexander invariant, an analog of the Alexander polynomial, and proved its relation with the $L^2$-torsion of the knot exterior.

In this article, we prove that the $L^2$-Alexander invariant for knots detects the unknot, in the following theorem.

\begin{thm} \label{main thm} [Main theorem] 
Let $K$ be a knot in $S^3$.
The $L^2$-Alexander invariant of $K$ is trivial, i.e.
 $\left (t \mapsto \Delta_K^{(2)}(t) \right ) = (t \mapsto 1)$, if and only if $K$ is the trivial knot.
\end{thm}

This theorem is proven by using the well-known fact (cf \cite{Mur}) that a knot exterior either has nonzero Gromov norm or is a graph manifold, and that in this second case the knot is obtained from the trivial knot by connected sums and cablings. In the first case, a theorem of Lück helps us conclude, and the second case is treated with help from the following connected sum and cabling formulas for the $L^2$-Alexander invariant.

\begin{thm} \label{formulas thm}
(1) The $L^2$-Alexander invariant is multiplicative under the connected sum of knots.

(2) The $L^2$-Alexander invariant satisfies the following cabling formula:

 if $S$ is the $(p,q)$-cable knot of companion knot $C$, then
$$\Delta_S^{(2)}(t) = \Delta_C^{(2)}(t^p) \max(1,t)^{(|p|-1) (|q|-1)}.$$
\end{thm}

These results were previously announced in \cite{BA}.

The article is organized as follows:
Section \ref{section preliminaries} reviews some well-known facts about knots, groups, and $L^2$-invariants, 
Sections \ref{section composite} and \ref{section cabling} prove the first and second parts of Theorem \ref{formulas thm}, 
Section \ref{section unknot} proves Theorem \ref{main thm}.
Section \ref{section} deals with the proof of the technical Proposition \ref{prop groupe cable}. Finally in section \ref{section open} we mention some open questions and research directions about the invariant.

\section*{Acknowledgements}

I would like to offer my special thanks to my advisor J\'er\^ome Dubois, who taught me all about $L^2$-invariants and always showed me the right mathematical direction to explore.

\section{Preliminaries} \label{section preliminaries}

\subsection{From knots to group presentations}

We choose an orientation for $S^3$.
All knots will be assumed oriented, and considered up to (orientation-preserving) isotopy in $S^3$. A link with $c \in \N$ components will be called a \textit{$c$-link}.

Let $K$ be an oriented knot in $S^3$, and $V(K)$ an open tubular neighbourhood of $K$. The exterior of $K$ is $M_K = S^3 \setminus V(K)$ and is a compact $3$-manifold with toroidal boundary. We fix a base point $pt$ in $M_K$. The orientation of $M_K$ comes from the one of $S^3$, and does not depend on the orientation of $K$.

Besides, since $K$ is oriented, there are, up to isotopy, unique simple closed curves $\mu_K$ and $\lambda_K$ on the $2$-torus $\partial M_K = \partial V(K)$ such that $\mu_K$ bounds a disk in $V(K)$ and $\lambda_K$ is homologous to $K$ in $V(K)$. We choose an orientation for these two curves such that the linking number between $\mu_K$ and $K$ and the intersection number between $\mu_K$ and $\lambda_K$ are both $+1$. We call $(\mu_K,\lambda_K)$ a preferred meridian-longitude pair for $K$. (Here we have used the notations and definitions of \cite{Tsau}).

Let us now consider the knot group $G_K = \pi_1(M_K, pt)$. 
We will call \textit{meridian loops} the elements of $G_K$ that are the homotopy classes of meridian curves.
The abelianization of $G_K$ is the infinite cyclic group. There are therefore exactly two surjective group homomorphisms from $G_K$ to $\Z$. We will write $\alpha_K\colon G_K\to \Z$ the one that sends meridian loops to $1$. Note that this choice depends strongly on the orientation of $K$.

When considering a group presentation $P = \langle g_1, \ldots , g_k | r_1 , \ldots , r_l \rangle$, it is usual to assimilate the combinatoric $(k+l)$-tuple and the generated group. In this article, we will use the first convention, and we would denote $Gr(P)$ the quotient of the free group $\F[g_1, \ldots, g_k ]$ by its normal subgroup generated by the free words $r_1, \ldots, r_l$.
We will say that a group \textit{$G$ admits the presentation $P= \langle g_1, \ldots , g_k | r_1 , \ldots , r_l \rangle$} when $G$ is isomorphic to $Gr(P)$, and we will assume that this isomorphism is implicit, or equivalently that we implicitly know which elements of $G$ are associated to $g_1, \ldots, g_k$.

For instance, the well-known Wirtinger process takes a regular diagram $D$ of a knot $K$ and gives a deficiency one group presentation $P$ of the knot group $G_K$, and the generators of $P$ all implicitly correspond to meridian loops in $G_K$; therefore they are all sent to the same image $1$ by the abelianization $\alpha_{K}$.

Let $p$ and $q$ be relatively prime integers, and let $V$ be a solid torus with a preferred meridian-longitude system (and thus an oriented core). The knot $T(p,q)$ on the boundary $\partial V$ of $V$ will denote the knot that wraps around $V$ $q$ times in the meridional direction and $p$ times in the longitudinal direction; it will be called the $(p,q)$-torus knot. Note that this follows the conventions of \cite{Rol} but not the ones of \cite{BZ} and \cite{Cro}, where the roles of $p$ and $q$ are reversed.

\subsection{Satellite knots}

Since we will use satellite and cable knots somewhat intensively in Section \ref{section cabling} and Section \ref{section}, we recall some definitions and fix some notations. We use the notations of \cite[Section 4]{Cro}.

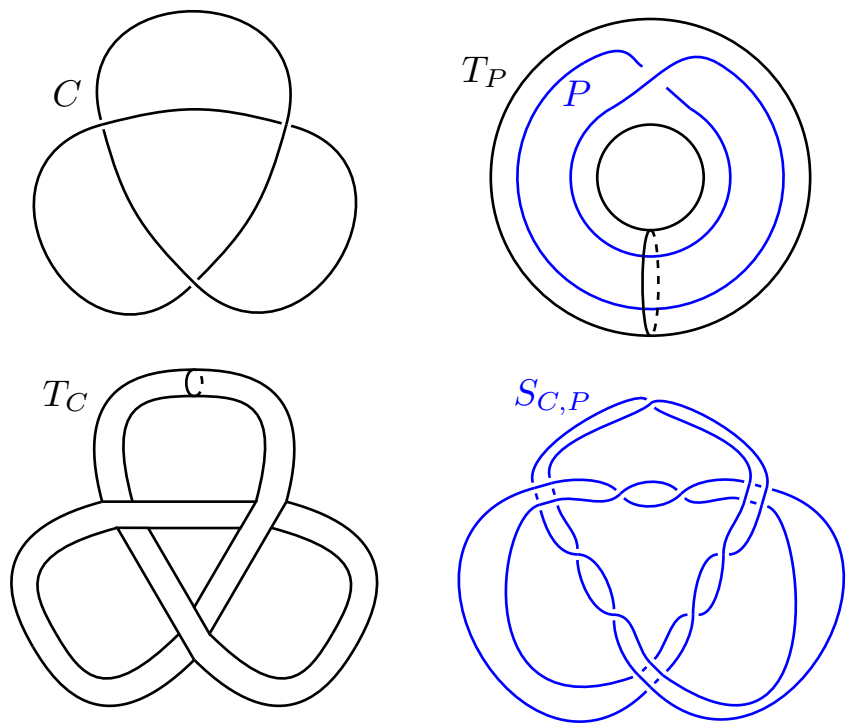
\begin{figure}[h]
\centering
\begin{tikzpicture} [every path/.style={string } , every node/.style={transform shape , knot crossing , inner sep=1.5 pt } ] 

\begin{scope}[xshift=-1cm,scale=0.7]
	\node[rotate=150] (tl) at (-1.73,1) {};
	\node[rotate=30] (tr) at (1.73, 1 ) { } ;
	\node[rotate=0] (b) at (0 ,-2) { } ;
	\coordinate (n) at (-2,1.2) ;
	\draw[scale=2] (n) node[above left]{$C$} ;

\draw (tl.center) .. controls (tl.16 south west) and (tr.16 north west) .. (tr) ;
\draw (tl) .. controls (tl.32 south east) and (tr.32 north east) .. (tr.center) ;
\draw (tl) .. controls (tl.16 north west) and (b.16 north west) .. (b.center) ;
\draw (tl.center) .. controls (tl.32 north east) and (b.32 south west) .. (b) ;
\draw (b) .. controls (b.16 north east) and (tr.16 south west) .. (tr.center) ;
\draw (b.center) .. controls (b.32 south east) and (tr.32 south east) .. (tr) ;
\end{scope}

\begin{scope}[xshift=5cm,scale=0.7]
	\coordinate (n) at (-1,1.2) ;
	\draw[scale=2] (n) node[blue, above left]{$P$} ;
	\coordinate (m) at (-2.6,1.5) ;
	\draw[scale=2] (m) node[above left]{$T_P$} ;
	\coordinate (O) at (0,0) ;
	\coordinate (A) at (0.75,1.3) ;
	\coordinate (B) at (1.25,2.17) ;
	\coordinate (C) at (-0.75,1.3) ;
	\coordinate (D) at (-1.25,2.17) ;	
	\coordinate (G) at (0,-1) ;
	\coordinate (H) at (0,-3) ;
	\coordinate (I) at (-0.15,2.1) ;
	\draw (O) circle (1) ;
	\draw (O) circle (3) ;	
	\draw[color=blue] (A) arc (60:-240:1.5) ;
	\draw[color=blue] (C) ..controls +(0.87,0.5) and +(0.8*-0.87,0.8*0.5).. (B) ;
	\draw[color=blue] (B) arc (60:-240:2.5) ;
	\draw[color=blue] (D) ..controls +(0.87,0.5) and +(-1*0.2,0.2*0.44)..(I)  ;
	\draw[color=blue] (A) -- ++(-0.9*0.5,0.5*0.8) ;
	\draw (G) ..controls +(-0.2,0) and +(-0.2,0).. (H) ;
	\draw [dashed] (G) ..controls +(0.2,0) and +(0.2,0).. (H) ;
\end{scope}

\begin{scope}[xshift=-1cm,yshift=-5cm,scale=0.7]
	\coordinate (n) at (-1.9,2.5) ;
	\draw[scale=2] (n) node[above left]{$ T_C $} ;
	\coordinate (O) at (0,0) ;
	\coordinate (A) at (-0.87,0.5) ;
	\coordinate (B) at (-1.73,1) ;
	\coordinate (C) at (-1.16,1) ;
	\coordinate (D) at (0.87,0.5) ;
	\coordinate (E) at (1.73,1) ;
	\coordinate (F) at (1.16,1) ;
	\coordinate (G) at (0,3) ;
	\coordinate (H) at (0,3.5) ;
	\draw (A) -- (D) -- (F) -- (B) ;
	\draw (F) ..controls +(0.5*1,0.5*3.73) and +(1,0).. (G) ;
	\draw (C) ..controls +(-0.5*1,0.5*3.73) and +(-1,0).. (G) ;
	\draw (E) ..controls +(0.6*1,0.6*3.73) and +(1,0).. (H) ;
	\draw (B) ..controls +(-0.6*1,0.6*3.73) and +(-1,0).. (H) ;
	\draw (G) ..controls +(-0.2,0) and +(-0.2,0).. (H) ;
	\draw [dashed] (G) ..controls +(0.2,0) and +(0.2,0).. (H) ;
\begin{scope}[rotate=120]
	\coordinate (O) at (0,0) ;
	\coordinate (A) at (-0.87,0.5) ;
	\coordinate (B) at (-1.73,1) ;
	\coordinate (C) at (-1.16,1) ;
	\coordinate (D) at (0.87,0.5) ;
	\coordinate (E) at (1.73,1) ;
	\coordinate (F) at (1.16,1) ;	
	\coordinate (G) at (0,3) ;
	\coordinate (H) at (0,3.5) ;
	\draw (A) -- (D) -- (F) -- (B) ;
	\draw (F) ..controls +(0.5*1,0.5*3.73) and +(1,0).. (G) ;
	\draw (C) ..controls +(-0.5*1,0.5*3.73) and +(-1,0).. (G) ;
	\draw (E) ..controls +(0.6*1,0.6*3.73) and +(1,0).. (H) ;
	\draw (B) ..controls +(-0.6*1,0.6*3.73) and +(-1,0).. (H) ;
\end{scope}
\begin{scope}[rotate=240]
	\coordinate (O) at (0,0) ;
	\coordinate (A) at (-0.87,0.5) ;
	\coordinate (B) at (-1.73,1) ;
	\coordinate (C) at (-1.16,1) ;
	\coordinate (D) at (0.87,0.5) ;
	\coordinate (E) at (1.73,1) ;
	\coordinate (F) at (1.16,1) ;	
	\coordinate (G) at (0,3) ;
	\coordinate (H) at (0,3.5) ;
	\draw (A) -- (D) -- (F) -- (B) ;
	\draw (F) ..controls +(0.5*1,0.5*3.73) and +(1,0).. (G) ;
	\draw (C) ..controls +(-0.5*1,0.5*3.73) and +(-1,0).. (G) ;
	\draw (E) ..controls +(0.6*1,0.6*3.73) and +(1,0).. (H) ;
	\draw (B) ..controls +(-0.6*1,0.6*3.73) and +(-1,0).. (H) ;
\end{scope}

\end{scope}

\begin{scope}[xshift=5cm,yshift=-5cm,scale=0.8,color=blue]
	\coordinate (n) at (-1,2) ;
	\draw[scale=1.75] (n) node[blue, above left]{$S_{C,P}$} ;
	\node[rotate=150] (g1) at (-1.63,1.18) {};
	\node[rotate=150] (g2) at (-1.91,1.1) {};
	\node[rotate=150] (g3) at (-1.83,0.82) {};
	\node[rotate=150] (g4) at (-1.55,0.9) {};
	\node[rotate=30] (d1) at (1.63, 1.18 ) { } ;
	\node[rotate=30] (d2) at (1.55, 0.9 ) { } ;	
	\node[rotate=30] (d3) at (1.83, 0.82 ) { } ;
	\node[rotate=30] (d4) at (1.91, 1.1 ) { } ;
	\node[rotate=0] (b1) at (0 ,-1.8) { } ;
	\node[rotate=0] (b2) at (-0.2 ,-2) { } ;	
	\node[rotate=0] (b3) at (0 ,-2.2) { } ;
	\node[rotate=0] (b4) at (0.2 ,-2) { } ;
	\node[rotate=0] (h) at (0 ,2.5) { } ;	
\draw (g1) .. controls (g1.8 south east) and (h.4 south west) .. (h.center) ;
\draw (g2) .. controls (g2.8 south east) and (h.4 north west) .. (h) ;
\draw (d4.center) .. controls (d4.8 north east) and (h.4 north east) .. (h.center) ;
\draw (d1.center) .. controls (d1.8 north east) and (h.4 south east) .. (h) ;

\draw (g2.center) .. controls (g2.32 north east) and (b3.32 south west) .. (b3) ;
\draw (g3.center) .. controls (g3.8 north east) and (-3,-3) .. (b2) ;

\draw (g2.center) -- (g1.center) ;
\draw (g3.center) -- (g4.center) ;
\draw (g1)-- (g4);
\draw (g2) -- (g3) ;

\draw (b3.center) .. controls (b3.32 south east) and (d4.32 south east) .. (d4) ;
\draw (b4.center) .. controls (3,-4) and (d3.8 south east) .. (d3) ;
\draw (d2.center) -- (d1.center) ;
\draw (d3.center) -- (d4.center) ;
\draw (d1)-- (d4);
\draw (d2) -- (d3) ;
\draw (b1.center) -- (b4.center) ;
\draw (b2.center) -- (b3.center) ;
\draw (b1)-- (b2);
\draw (b3) -- (b4) ;

	\node (tm1) at (-0.5,1) {};
	\node (tm2) at (0.5,1) {};

\draw (g1.center) .. controls (g1.4 south west) and (tm1.4 north west) .. (tm1) ;
\draw (tm1) .. controls (tm1.4 south east) and (tm2.4 south west) .. (tm2.center) ;
\draw (tm2.center) .. controls (tm2.4 north east) and (d1.4 north west) .. (d1) ;

\draw (g4.center) .. controls (g4.4 south west) and (tm1.4 south west) .. (tm1.center) ;
\draw (tm1.center) .. controls (tm1.4 north east) and (tm2.4 north west) .. (tm2) ;
\draw (tm2) .. controls (tm2.4 south east) and (d2.4 north west) .. (d2) ;

	\node (tg1) at (-1.2,0) {};
	\node (tg2) at (-0.6,-1) {};

\draw (g4) .. controls (g4.4 north west) and (tg1.4 north) .. (tg1) ;
\draw (tg1) .. controls (tg1.4 south) and (tg2.4 west) .. (tg2.center) ;
\draw (tg2.center) .. controls (tg2.4 east) and (b1.4 north west) .. (b1.center) ;

\draw (g3) .. controls (g3.4 north west) and (tg1.4 west) .. (tg1.center) ;
\draw (tg1.center) .. controls (tg1.4 east) and (tg2.4 north) .. (tg2) ;
\draw (tg2) .. controls (tg2.4 south) and (b2.4 north west) .. (b2.center) ;

	\node (td1) at (1.2,0) {};
	\node (td2) at (0.7,-1) {};

\draw (d3.center) .. controls (d3.4 south west) and (td1.4 east) .. (td1) ;
\draw (td1) .. controls (td1.4 west) and (td2.4 north) .. (td2.center) ;
\draw (td2.center) .. controls (td2.4 south) and (b4.4 north east) .. (b4) ;

\draw (d2.center) .. controls (d2.4 south west) and (td1.4 north) .. (td1.center) ;
\draw (td1.center) .. controls (td1.4 south) and (td2.4 east) .. (td2) ;
\draw (td2) .. controls (td2.4 west) and (b1.4 north east) .. (b1) ;

\end{scope}

\end{tikzpicture}

\caption{The $(2,-1)$-cabling of the trefoil knot} \label{fig cable trefoil}
\end{figure}

Let $C$ be a non-trivial knot in $S^3$ (it will be called the \textit{companion knot}).

We consider $P$ a knot inside an open solid torus $T_P$, $T_P$ being also embedded in $S^3$ ($P$ will be called the \textit{pattern knot}). 
We choose an orientation for the core of $T_P$.
We assume that
$P$ meets every meridional disk of $T_P$.
We let $n_P \in \Z$ denote the linking number between $P$ and a preferred meridian curve of $\partial T_P$ (assumed to be positively oriented with the orientation of the core of $T_P$). Note that preferred longitude curves of $T_P$ have zero linking number with the core of $T_P$ and follow the same direction.

Let $T_C$ be an open tubular neigborhood of $C$ (its core having the same orientation as $C$).
Notice that a preferred longitude curve of $T_C$ has zero linking number with $C$. Thus the homotopy class in $G_C$ of such a curve is sent to zero by the abelianization $\alpha_C$.

Let $h_{PC}\colon T_P \rightarrow T_C$ be an orientation-preserving homeomorphism between the two solid tori. We also assume that $h_{PC}$ sends a preferred meridian-longitude pair of $T_P$ to a preferred meridian-longitude pair of $T_C$.

Then $S_{C,P}:= h_{PC}(P)$ is a knot in $S^3$ and is called \textit{the satellite knot of companion $C$ and pattern $P$}.

Let us mention two particular cases of satellites: the cables and the Whitehead doubles.

If $P$ is a torus knot $T(p,q)$ (naturally defined on the boundary of a solid sub-torus of $T_P$), then we call $S_{C,P}$ a \textit{cable knot}, or \textit{the $(p,q)$-cable of $C$}. In this case $n_P = p$.

Figure \ref{fig cable trefoil} gives an example of $S_{C,P}$ when $C$ is the trefoil knot and $P$ is the torus knot pattern $T(2,-1)$. The orientations are not marked but should be clear.

\begin{figure}[h]
\centering
\begin{tikzpicture} [every path/.style={string } , every node/.style={transform shape , knot crossing , inner sep=1.5 pt } ] 
\begin{scope}[xshift=5cm,scale=1]
	\coordinate (n) at (-1,1.2) ;
	\coordinate (m) at (-2.6,1.5) ;
	\coordinate (O) at (0,0) ;
	\node (A) at (0.75,1.3) {};
	\node (B) at (1.25,2.17) {};
	\node (b) at (0,1.7) {};
	\node (h) at (0,2.2) {};
	\node (C) at (-0.75,1.3) {};
	\node (D) at (-1.25,2.17) {};	
	\coordinate (G) at (0,-1) ;
	\coordinate (H) at (0,-3) ;
	\coordinate (I) at (-0.15,2.1) ;
	\draw (O) circle (1) ;
	\draw (O) circle (3) ;	
	\draw[<-][color=blue] (A) arc (60:-240:1.5) ;
	\draw[->][color=blue] (B) arc (60:-240:2.5) ;
	\draw[color=blue] (A.center) .. controls (A.2 north west) and (b.4 south east) .. (b) ;
	\draw[color=blue] (b) .. controls (b.4 north west) and (h.4 south west) .. (h.center) ;
	\draw[->][color=blue] (h.center) .. controls (h.4 north east) and (B.2 north west) .. (B.center) ;
	\draw[<-][color=blue] (C.center) .. controls (C.2 north east) and (b.4 south west) .. (b.center) ;
	\draw[color=blue] (b.center) .. controls (b.4 north east) and (h.4 south east) .. (h) ;
	\draw[color=blue] (h) .. controls (h.4 north west) and (D.2 north east) 
.. (D.center) ;
	\draw (G) ..controls +(-0.2,0) and +(-0.2,0).. (H) ;
	\draw [dashed] (G) ..controls +(0.2,0) and +(0.2,0).. (H) ;
\end{scope}
\end{tikzpicture}

\caption{The Whitehead double pattern} \label{fig whitehead pattern}
\end{figure}
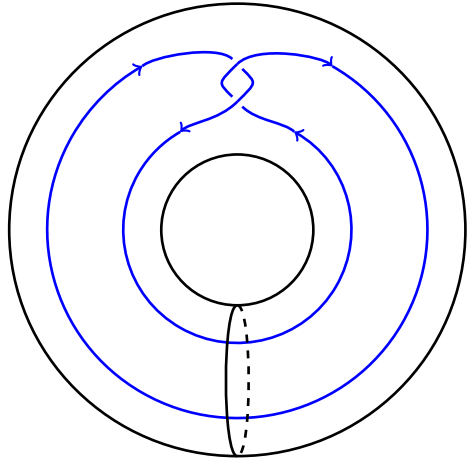

If $P$ is the pattern in Figure \ref{fig whitehead pattern}, called the \textit{Whitehead double pattern}, then $n_P=0$ and $S_{C,P}$ is called the \textit{Whitehead double of $C$}.

\subsection{Connected sum, cabling, and groups}

Here we state some useful results about how the connected sum and cabling operations affect the knot groups.

\begin{prop} \label{prop Amalgame}

Let $K_1$, $K_2$ be two knots and $K$ their connected sum. We let $G_1, G_2, G$ denote their respective knot groups. Then
$G_1$ and $G_2$ have Wirtinger presentations $P_1 = \langle x_1 , \ldots , x_k | r_1 , \ldots, r_{k-1}   \rangle$, 
$P_2 = \langle y_1 , \ldots , y_l | s_1 , \ldots, s_{l-1}   \rangle$  such that
$$P = \langle x_1 , \ldots , x_k, y_1 , \ldots , y_l | r_1 , \ldots, r_{k-1}, s_1, \ldots, s_{l-1}, x_k y_l^{-1}   \rangle$$
is a Wirtinger presentation of $G$.
\end{prop}

This proposition is a consequence of the Seifert-van Kampen theorem. The associated partition (two open sets $U_1$ and $U_2$, their union and their intersection) can be seen in Figure \ref{figure vankampen partition}.

\begin{figure}[h]
\centering

\begin{tikzpicture}[every path/.style={string ,black} , every node/.style={transform shape , knot crossing , inner sep=1.5 pt } ]

\begin{scope}[xshift=0cm,scale=0.7]
\draw [gray!20,fill=gray!20] (-3,0.6) -- (-3,4.5) -- (2.5,4.5) -- (2.5,1.1) -- (2,0.6) -- cycle ;
\draw [gray!30,fill=gray!30] (-3,-0.6) -- (-3,-4.5) -- (2.5,-4.5) -- (2.5,-0.1) -- (2,-0.6) -- cycle ;
\draw [gray!50,fill=gray!50] (-3,0.6) -- (-3,-0.6) -- (2,-0.6) -- (2.5,-0.1) -- (2.5,1.1) -- (2,0.6) -- cycle ;

\draw (1,0) node[right]{$pt$} ;
\draw (1,0) node {$\bullet$} ;
\draw (0,4.4) node[right]{$\infty$} ;
\draw (0,4.4) node {$\bullet$} ;
\draw (0,-4.4) node[right]{$\infty$} ;
\draw (0,-4.4) node {$\bullet$} ;
\coordinate (a) at (0,0.6) ;
\draw (a) node[above right]{$A$} ;
\draw (a) node {$\bullet$} ;
\coordinate (b) at (0,-0.6) ;
\draw (b) node[above right]{$B$} ;
\draw (b) node {$\bullet$} ;
\draw (a) -- (b) ;
\draw (a) -- ++(2,0) -- ++(0.5,0.5) ;
\draw (a) -- ++(-3,0) -- ++(0.5,0.5) ;
\draw (b) -- ++(2,0) -- ++(0.5,0.5) ;
\draw (b) -- ++(-3,0) -- ++(0.5,0.5) ;
\draw [<->] (-1,0.5)--(-1,-4.5);
\draw [<->] (-2,-0.5)--(-2,4.5);
\draw (-2,2) node[left]{$U_1$} ;
\draw (-1,-2) node[left]{$U_2$} ;
	\begin{scope}[rotate=-90, xshift=-2.5cm]
		\node[rotate=0] (d) at (-2,0) { } ;	
		\node[rotate=45] (a) at (-0.7,0) { } ;		
		\node[rotate=0] (b) at (0,0.5) { } ;		
		\node[rotate=45] (c) at (0.7,0) { } ;		
		\node[rotate=0] (f) at (2,0) { } ;		
	\draw (d) .. controls (-1,0) and (a.4 north west) .. (a.center) ;
	\draw (b) .. controls (b.4 south west) and (a.4 south east) .. (a.center) ;
	\draw (b) .. controls (b.4 north east) and (c.4 north east) .. (c.center) ;
	\draw (a) .. controls (a.4 south west) and (c.4 south west) .. (c.center) ;
	\draw (a) .. controls (a.4 north east) and (b.4 north west) .. (b.center) ;
	\draw (c) .. controls (c.4 north west) and (b.4 south east) .. (b.center) ;
	\draw (c) .. controls (c.4 south east) and (1,0) .. (f) ;
	\end{scope}

	\begin{scope}[rotate=-90, xshift=2.5cm]
		\node[rotate=0] (d) at (-2,0) { } ;	
		\node[rotate=45] (a) at (-1,0) { } ;		
		\node[rotate=0] (b) at (-0.2,0.4) { } ;		
		\node[rotate=0] (e) at (0.2,-0.4) { } ;		
		\node[rotate=45] (c) at (1,0) { } ;		
		\node[rotate=0] (f) at (2,0) { } ;		
	\draw (d) .. controls (-1,0) and (a.4 north west) .. (a) ;
	\draw (b.center) .. controls (b.4 south west) and (a.4 south east) .. (a) ;
	\draw (b.center) .. controls (b.4 north east) and (c.4 north east) .. (c) ;
	\draw (e.center) .. controls (e.4 south east) and (c.4 south west) .. (c) ;
	\draw (b) .. controls (b.4 south east) and (e.4 north west) .. (e.center) ;
	\draw (b) .. controls (b.4 north west) and (a.4 north east) .. (a.center) ;
	\draw (e) .. controls (e.4 south west) and (a.4 south west) .. (a.center) ;
	\draw (e) .. controls (e.4 north east) and (c.4 north west) .. (c.center) ;
	\draw (c.center) .. controls (c.4 south east) and (1,0) .. (f) ;
	\end{scope}
\end{scope}
\end{tikzpicture}

\caption{The Seifert-van Kampen partition, $K_1$ is the trefoil, $K_2$ the figure-eight} \label{figure vankampen partition}
\end{figure}
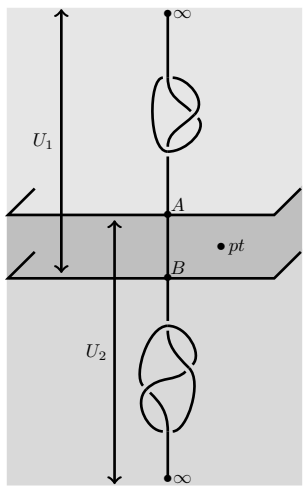

\begin{prop} \label{prop groupe cable}
Let us consider the $(p,q)$-cable knot $S$ of companion $C$. 

(1) There exists $P_C = \langle a_1 , \ldots , a_k | r_1 , \ldots, r_{k-1}   \rangle$ a Wirtinger presentation of $G_C$  such that
$$P_S = \langle a_1 , \ldots , a_k, x, \lambda | r_1 , \ldots, r_{k-1}, x^p  a_k^{-q} \lambda^{-p} ,
\lambda^{-1} W(a_i) 
   \rangle$$
is a presentation of $G_S$, with  $x$ and $\lambda$ the homotopy classes of the core and a longitude of $T_C$, and $W(a_i)$ a word in the $a_1, \ldots ,a_k$.

(2) Furthermore, $\alpha_{S}(x) = q$, $\alpha_{S}(\lambda) =0$ and $\alpha_{S}(a_i) = p$, for $i=1, \ldots , k$.

\end{prop}

We give a detailed proof of this proposition in Section \ref{section}. Note that this result can be found in a different flavour in \cite[Section 4.12]{BZ}.

\bigskip

Both following propositions are consequences of \cite[Theorem 4.3]{magnus}, and will be useful for induction properties. Note that the proof of Proposition \ref{prop sat inj} also uses \cite[Proposition 3.17]{BZ}.

\begin{prop} \label{prop somme inj}
If $K$ is the connected sum of the knots $K_1$ and $K_2$, and $G, G_1, G_2$ are their respective groups, then there are injective group homomorphisms $ G_1 \hookrightarrow G$ and $G_2 \hookrightarrow G$.
\end{prop}

\begin{prop} \label{prop sat inj}
If $S$ is the satellite knot obtained from the companion $C$ and the pattern $P$, then there is an injective group homomorphism $ G_C \hookrightarrow G_S$. 
\end{prop}

\subsection{Fox calculus}

Let $ P = \big\langle g_1, \ldots, g_k \, \big| \, r_1, \ldots r_{l} \big\rangle $ be a presentation of a finitely presented group $G$.
If $w$ is an element of the free group 
$\mathbb{F} [ g_1 , \ldots , g_k ]$
on the generators $g_i$, we note $\overline{w}$ the element of $G$ that is the image of $w$ by the composition of the quotient homomorphism (quotient by the normal subgroup $\langle r_j \rangle$ generated by $r_1, \ldots, r_l$) and the implicit group isomorphism between this quotient $Gr(P)$ and $G$. 
To simplify the notations in the sequel, we will often write an element of $G$ $a$ instead of $\overline{a}$ when there is no ambiguity.
We note the corresponding ring morphisms similarly: if $w \in  \C \left [ \mathbb{F} [ g_1 , \ldots , g_k ] \right ]$ then its quotient image is noted $\overline{w} \in \C[G] $.

The \textit{Fox derivatives associated to the presentation $P$} are the linear maps \\
$ {\dfrac{\partial}{\partial g_i}\colon \C \left [ \mathbb{F} [ g_1 , \ldots , g_k ] \right ]  \longrightarrow \C \left [ \mathbb{F} [ g_1 , \ldots , g_k ] \right ]}  $ for $i= 1 , \ldots , k$, defined by induction in the following way:

$ \dfrac{\partial}{\partial g_i} (1) = 0 , 
\ \dfrac{\partial}{\partial g_i} (g_j) = \delta_{i,j} , 
\ \dfrac{\partial}{\partial g_i} (g_j^{-1}) = - \delta_{i,j} g_j^{-1}$ (where $\delta_{i,j}$ is $1$ when $i=j$ and $0$ when $i \neq j$)  and 
for all  $u , v \in \mathbb{F} [ g_1 , \ldots , g_n ] , \ \dfrac{\partial}{\partial g_i} (u v) = \dfrac{\partial}{\partial g_i} (u) + u \dfrac{\partial}{\partial g_i} (v)$.

\begin{defn}We call 
$F_P = \left ( \overline{ \left ( \dfrac{\partial r_j}{\partial g_i} \right )} \right )_{1 \leqslant i \leqslant k , 1 \leqslant j \leqslant l } \in M_{k,l}(\C \left [ G \right ] )$  the \textit{Fox matrix} of the presentation $P$. 

Let us assume $l =k-1$, i.e. \textbf{$P$ has deficiency one}.
For $i = 1 , \ldots  , k$, 
$F_{P,i} \in M_{k-1,k-1}(\C \left [ G \right ] )$ is defined as the matrix obtained from $F_P$ by deleting its $i$-th row.
\end{defn}

We will sometimes use the following notation, to  \guillemotleft remember the coordinates\guillemotright : 
$$F_P=\kbordermatrix{
 & r_1 & \ldots & r_{l} \\
x_1 & & & \\
\vdots & &  \left (\overline{ \left ( \dfrac{\partial r_j}{\partial g_i} \right )}  \right )_{i,j} &  \\
x_{k} & & & 
}$$

\bigskip

\begin{ex}
$ P = \big\langle a, b \, \big| \, a b a b^{-1} a^{-1} b^{-1} \big\rangle $ is a presentation of the group of the trefoil knot. Let us denote $ r = a b a b^{-1} a^{-1} b^{-1}$. One has:

$\dfrac{\partial r}{\partial a} = 1 + a b  - a b a b^{-1} a^{-1}$ so 
$ \overline{ \left ( \dfrac{\partial r}{\partial a} \right )} =  1 - \overline{b} + \overline{ab}$.

$\dfrac{\partial r}{\partial b} = a - a b a b^{-1} - a b a b^{-1} a^{-1} b^{-1}$ so 
$ \overline{ \left ( \dfrac{\partial r}{\partial b} \right )} =  -1 + \overline{a} - \overline{ba}$.
 
Thus $ F_P = \begin{pmatrix}
1 - b + ab \\
-1 +a - ba
\end{pmatrix}$
, $ F_{P,1} = (1 - b + a b)$ and $ F_{P, 2 } = (- 1 + a - b a)$.
\end{ex}

\subsection{$L^2$-invariants}

Let $G$ be a countable discrete group (a knot group, for example). In the following, every algebra will be a $\C$-algebra.

We denote $ l^2(G):= \left \{ \sum_{g \in G} \lambda_g [g] \ | \  \lambda_g \in \mathbb{C} , \sum_{g \in G} | \lambda_g |^2 < \infty \right \}$ the Hilbert space of square-summable complex functions on the group $G$.

It is the completion of the vector space $ \mathbb{C} [G] = \bigoplus_{g \in G} \mathbb{C}[g] $ (which is also an algebra) for the scalar product:
$$ \left \langle \sum_{g \in G} \lambda_g [g] , \sum_{g \in G} \mu_g [g] \right \rangle:= \sum_{g \in G} \lambda_g \overline{\mu_g}. $$

We denote
$\mathcal{B}( l^2(G))$ the algebra of operators on $l^2(G)$ that are continuous (or equivalently, bounded) for the operator norm.

To any $h \in G$ we associate a \textit{left-multiplication} 
$L_h\colon l^2(G) \rightarrow l^2(G)$ defined by
$$ L_h \left ( \sum_{g \in G} \lambda_g [g] \right ) = \sum_{g \in G} \lambda_g [hg] = \sum_{g \in G} \lambda_{h^{-1}g} [g]$$
and a \textit{right-multiplication} $R_h\colon l^2(G) \rightarrow l^2(G)$ defined by
$$ R_h \left ( \sum_{g \in G} \lambda_g [g] \right ) = \sum_{g \in G} \lambda_g [gh] = \sum_{g \in G} \lambda_{gh^{-1}} [g].$$
Both $L_h$ and $R_h$ are isometries, and therefore belong to $\mathcal{B}( l^2(G))$.

We will use the same notation for right-multiplications by elements of the complex group algebra $\C [G]$:
$$ R_{ \sum_{i=1}^{k} \lambda_i [g_i] }:= \sum_{i=1}^{k} \lambda_i R_{g_i} \in \mathcal{B}( l^2(G)).$$

We will also use this notation to define a right-multiplication by a matrix $A$ with coefficients in $\C [G]$, $p$ rows and $q$ columns, in the following way:

If $A = \left ( a_{i,j} \right )_{1 \leqslant i \leqslant p, 1 \leqslant j \leqslant q} \in M_{p,q}(\C[G])$, then 
$$R_A:= \left (R_{a_{i,j}}\right )_{1 \leqslant i \leqslant p, 1 \leqslant j \leqslant q} 
\in \mathcal{B}( l^2(G)^{\oplus q}; l^2(G)^{\oplus p} ).$$

We write $\mathcal{N}(G)$ the algebraic commutant of $ \{ L_g ; g \in G \} $ in $\mathcal{B}( l^2(G))$. It will be called the \textit{von Neumann algebra of the group $G$}.

Let us remark that $R_g \in \mathcal{N}(G)$ for all $g$ in $G$.

The \textit{trace} of an element $\phi$ of $\mathcal{N}(G)$ will be defined by  $$tr_{\mathcal{N}(G)}(\phi):= \left \langle \phi ([e]) , [e] \right \rangle$$ where $ e $ is the neutral element of $G$.
This induces a trace on the $M_{n,n}(\mathcal{N}(G))$ for $n \geq 1$ by summing up the traces of the diagonal elements. We will note this new trace $tr_{\mathcal{N}(G)}$ as well.

We will call a \textit{finite type $\mathcal{N}(G)$-Hilbert module} (or simply \textit{$\mathcal{N}(G)$-module} in the following) any Hilbert space $V$ on which there is a left $G$-action by isometries, and 
such that there exists a positive integer $m$ and an embedding $\phi$ of $V$ into $\bigoplus_{i=1}^m \ell^2(G)$ (an embedding meaning here a linear isometrical injective $G$-equivariant map, where the left $G$-action on $\bigoplus_{i=1}^m \ell^2(G)$ is by left-multiplication coordinate by coordinate).

The \textit{von Neumann dimension} of such a ${\mathcal N}(G)$-module $V$ is defined as the trace of the projection:
$$\dim_{\mathcal{N}(G)}(V) = \mathrm{tr}_{\mathcal{N}(G)}(\mathrm{pr}_{\phi(V)}) \in \R_{\geqslant 0},$$
where
\[
\mathrm{pr}_{\phi(V)} \colon \bigoplus_{i=1}^k \ell^2(G) \to \bigoplus_{i=1}^k \ell^2(G) 
\]
is the orthogonal projection onto $\phi(V)$. The von Neumann dimension does not depend on the embedding of $V$ into the finite direct sum of copies of $\ell^2(G)$. 

If $U$ and $V$ are $ \mathcal{N}(G)$-modules, we will call $ f\colon U \rightarrow V$ a \textit{map of  $\mathcal{N}(G)$-modules} if $f$ is a linear $G$-equivariant map, bounded for the respective scalar products of $U$ and $V$.

\bigskip

Let us now write a little about induction.

Let $ i\colon H \hookrightarrow G$ be an injective group homomorphism. To simplify notations, we will also call $i$ the inducted algebra homomorphism on $\C[H]$ and matrices over $\C[H]$, and the isometric injection on $\ell^2(H)$. Let $M$ be a $\mathcal{N}(H)$-module. 
Then, according to \cite[Section 1.1.5]{Luc02}, we can construct an induction covariant functor $i_*$ from the category ($\mathcal{N}(H)$-modules, maps of $\mathcal{N}(H)$-modules) to 
($\mathcal{N}(G)$-modules, maps of $\mathcal{N}(G)$-modules),
such that $i_*(\ell^2(H)) = \ell^2(G)$.

\begin{rem}
Let us now recall a well-known property of the induction functor, summarized in part (1) of Proposition \ref{prop i*}:
 
If $w \in \C[H]$, then $R_w$ is $H$-equivariant because right-multiplications commute with left-multiplications. Thus it is a map of $N(H)$-modules.

Take $(m_j)$ a square-summable family of $\ell^2(H)$, and let $x = \sum_{j \in G / i(H)} L_{g_j} i(m_j)$ be a typical element of $i_*\ell^2(H) = \ell^2(G)$.

Then, since left-multiplications commute with right-multiplications, 
\begin{align*}
i_*R_w(x) &=  \sum_{j \in G / i(H)} L_{g_j} i(R_w(m_j)) =  \sum_{j \in G / i(H)} L_{g_j} R_{i(w)}i(m_j) \\ 
&= R_{i(w)} \sum_{j \in G / i(H)} L_{g_j} i(m_j) 
= R_{i(w)} (x).
\end{align*}

 Thus $i_*R_w = R_{i(w)}$.
\end{rem}

\smallskip

The following properties of this induction functor will be used in this paper:

\begin{prop} \label{prop i*}
(1) Let $w \in \C[H]$ and $R_w\colon \ell^2(H) \to \ell^2(H)$ be the corresponding right multiplication. Then $i_*R_w = R_{i(w)}$.

A similar result stands for matrices over $\C[H]$.

(2) If the map of $\mathcal{N}(H)$-modules  $f\colon M \to N$ is injective (resp. surjective) then $i_*f\colon i_*M \to i_*N$ is also injective (resp. surjective).

(3) If $M$ is a $\mathcal{N}(H)$-module, then
$dim_{\mathcal{N}(G)}(i_*M)=dim_{\mathcal{N}(H)}(M)$.

\end{prop}

\begin{rem}
For any $\phi \in \mathcal{N}(H)$, $i_*\phi$ is in $\mathcal{N}(G)$, because commuting with the left multiplications is the same as being equivariant for the group action.
\end{rem}

\subsection{The Fuglede-Kadison determinant}

\begin{defn}
Let $G$ be a finitely generated group and $U, V$ be two ${\mathcal N}(G)$-modules. Let $f \colon U \to V$ be a map of $\mathcal{N}(G)$-modules.
The \textit{spectral density of $f$} is the map $ \lambda \in \R_{\geqslant 0} \mapsto F(f)(\lambda)$ defined by:
\[
F(f)(\lambda):= \sup \{
 \dim_{\mathcal{N}(G)} (L) | L \in \mathcal{L}(f,\lambda) \}
 \]
where $\mathcal{L}(f,\lambda)$ is the set of sub-$\mathcal{N}(G)$-modules of $U$ on which the restriction of $f$ has a norm smaller than or equal to $\lambda$.

Let us remark that $F(f)(\lambda)$ is monotonous and right-continuous, and so defines a measure $dF(f)$ on the Borel set of $\R_{\geqslant 0}$ solely determined by the \\ $dF(f)(]a,b]) = F(f)(b)-F(f)(a)$ for all $a<b$.
\end{defn}

\begin{rem} \label{rem bornes F(f)}
Note that $\mathcal{L}(f,0)$ is the set of sub-$\mathcal{N}(G)$-modules of $Ker(f)$, and \\
$F(f)(0) = \dim_{\mathcal{N}(G)}(Ker(f))$. 

For any $\lambda \geqslant \| f\|$, $\mathcal{L}(f,\lambda)$ is the set of sub-$\mathcal{N}(G)$-modules of $U$, and \\
 $F(f)(\lambda) = \dim_{\mathcal{N}(G)}(U)$. 
\end{rem}

\begin{rem}
For all $\lambda$, $F(f)(\lambda) = F(f^*f)(\lambda^2) = F(|f|)(\lambda)$ where $f^*f\colon U \to U$ is a positive operator and $|f|$ is its square root.

We can thus think with positive operators and observe that $dF(f)$ measures the  \guillemotleft  density of eigenvalues\guillemotright . If $\lambda$ is atomic then $dF(f)(\lambda)$ is the von Neumann dimension of the eigenspace associated to $\lambda$.
\end{rem}

\begin{defn}
The \textit{Fuglede-Kadison determinant of $f$} is defined by:
\begin{equation*}\label{detFK}
{det}_{\mathcal{N}(G)}(f):= \exp \left ( \int_{0^+}^\infty \ln(\lambda) \, dF(f)(\lambda) \right )
\end{equation*}
if $\int_{0^+}^\infty \ln(\lambda) \, dF(f)(\lambda) > -\infty$\,; if not, $det_{\mathcal{N}(G)}(f) = 0$. 

When  $\int_{0^+}^\infty \ln(\lambda) \, dF(f)(\lambda) > -\infty$, we say that \emph{$f$ is of determinant class}.\\
\end{defn}

Here are several properties of the determinant we will use in the rest of this paper. The proofs can be found in \cite{Luc02}.

\begin{prop} \label{prop operations det}
(1) $det_{\mathcal{N}(G)} (0\colon U \to V) = 1$.

(2) For every nonzero complex number $\lambda$, $det_{\mathcal{N}(G)} (\lambda Id_U) = | \lambda |$.

(3) For all $f,g$ maps of $\mathcal{N}(G)$-modules, $$
det_{\mathcal{N}(G)} \left (
\begin{pmatrix}
f & 0 \\ 0 & g
\end{pmatrix}
\right ) = det_{\mathcal{N}(G)}(f) \cdot det_{\mathcal{N}(G)}(g).$$

(4) For $f\colon U \to V$ and $g\colon V \to W$ both injective maps of $\mathcal{N}(G)$-modules, 
$$ det_{\mathcal{N}(G)} (g  \circ f ) =
det_{\mathcal{N}(G)}(g) \cdot  det_{\mathcal{N}(G)}(f).$$

(5) Let $f_1\colon U_1 \to V_1$, $f_2\colon U_2 \to V_2$ and $f_3\colon U_2 \to V_1$ be maps of $\mathcal{N}(G)$-modules, such that $f_1$ and $f_2$ are injective. Then $$
det_{\mathcal{N}(G)} \left (
\begin{pmatrix}
f_1 & f_3 \\ 0 & f_2
\end{pmatrix}
\right ) = det_{\mathcal{N}(G)}(f_1) \cdot  det_{\mathcal{N}(G)}(f_2).$$

(6) Let $i\colon H \hookrightarrow G$ be an injective group homomorphism.
Let $M$ and $N$ be two $ \mathcal{N}(H)$-modules and $f\colon M \rightarrow N$ be a map of  $\mathcal{N}(H)$-modules. Then
$$ det_{\mathcal{N}(G)}(i_*(f)) = det_{\mathcal{N}(H)}(f).$$ 

\end{prop}

\begin{prop} \label{prop id - tg}
Let $g \in G$ be of infinite order, let $t \in \C$, then $ Id - t R_g$ is injective and 
$$
det_{\mathcal{N}(G)} ( Id - t R_g) = \max ( 1 , |t|).
$$
\end{prop}

The proof of this proposition can be found in \cite[Proposition 3.2, Remark 3.3]{LZ06a}. We offer the following proof for completeness, since the case $|t| \neq 1$ was not studied in detail in \cite{LZ06a}.

\begin{proof} Let us note $f=Id - t R_g\colon \ell^2(G) \to \ell^2(G)$.

The assertion is clearly true for $t=0$. 

We will now use the following known way of computing Fuglede-Kadison determinants, proved in \cite{CFM}: if $f_u$, $u \in [0;1]$, is a norm-continuous piecewise $C^1$ path in $GL(U)$ (the group of all invertible maps of $\mathcal{N}(G)$-modules $f\colon U \to U$) then $$
\dfrac{\det_{\mathcal{N}(G)}(f_1)}{\det_{\mathcal{N}(G)}(f_0)} =
\exp\left ( \int_0^1 Re \left ( tr_{\mathcal{N}(G)} \left (f_u^{-1} \circ \dfrac{\partial f_u}{\partial u}\right ) \right ) du \right ).
$$

Let us assume $0 < |t| <1$. Then we set $f_u = Id - u t R_g$ for $u \in [0;1]$. We have $f_0 = Id$, $f_1=f$, $\dfrac{\partial f_u}{\partial u}= - t R_g$, and since $R_g$ is unitary and $|t| <1$, 
$$f_u^{-1} = Id + \sum_{i=1}^{\infty} (u t)^i R_{g^i} .$$
 Thus the previous formula gives:
$$
det_{\mathcal{N}(G)} (f) = \exp \left ( \int_0^1 Re \left (
tr_{\mathcal{N}(G)} \left ( -t R_g - t \sum_{i=1}^{\infty}(ut)^i R_{g^{i+1}} \right ) \right ) dt \right ) 
$$
therefore $ \det_{\mathcal{N}(G)}(f) = \exp \left ( \int_0^1 Re (0) dt \right ) = 1
$
since $g$ is of infinite order. Hence the assertion is proven.

Now let us assume that $|t| >1$. Then $f = (-t R_g) \circ h$ where $h = Id - t^{-1}R_{g^{-1}}$. 
According to the previous case, $\det_{\mathcal{N}(G)}(h) = 1$. Besides, $(-t R_g)$ and $h$ are invertible and $\det_{\mathcal{N}(G)}(R_g) = \det_{\mathcal{N}(G)} (R_g^*R_g)^{\frac{1}{2}} = \det_{\mathcal{N}(G)}(Id)^{\frac{1}{2}} = 1$ (according to \cite[Lemma 3.15 (4)]{Luc02}), thus by Proposition \ref{prop operations det} (2) and (4), $\det_{\mathcal{N}(G)}(f) = |t|$.

Finally, if $|t|=1$, let us show that $f$ is injective. Let $a= \sum_{\gamma \in G} a_{\gamma} [\gamma] \in \ell^2(G)$ such that $f(a)=0$.
Then $ \sum_{\gamma \in G} (a_{\gamma}-t a_{\gamma g^{-1}}) [\gamma] = 0$, therefore for all $\gamma$ in $G$, all the $t^i a_{\gamma g^{-i}}$ are equal. Hence if $a$ was nonzero, there would be a nonzero coefficient $a_{\gamma}$, but that would imply that $$\| a \|^2 > \sum_{i \in \Z} | a_{\gamma g^i}|^2 =
\sum_{i \in \Z} |t^i a_{\gamma}|^2 =
 +\infty$$ since $|t|=1$. Thus $a = 0$ and $f$ is injective. 

Now, by  Lemma 3.15 (4)-(5) in \cite{Luc02},  
 $ \det_{\mathcal{N}(G)}(f) = \underset{\epsilon \to 0_+}{\lim}
 \det_{\mathcal{N}(G)} (f^*f + \epsilon Id)^{1/2}$
  and, by taking  $t_\epsilon = 1 + \frac{1}{2} \epsilon + \sqrt{\epsilon + \frac{1}{4}\epsilon^2}$,
\begin{align*}
f^*f + \epsilon Id = (2+ \epsilon)Id - R_g - R_g^*
  &= \dfrac{2+ \epsilon}{1 + t_\epsilon^2}\left ((1+t_\epsilon^2)Id - t_\epsilon R_g - t_\epsilon R_g^* \right ) \\
  &= \dfrac{2+ \epsilon}{1 + t_\epsilon^2} (Id - t_\epsilon R_g)^* (Id - t_\epsilon R_g).
\end{align*}

  Since $t_\epsilon >1$, we thus have $ \det_{\mathcal{N}(G)} ( Id - t_\epsilon R_g ) = \max ( 1, t_\epsilon) \underset{\epsilon \to 0_+}{\longrightarrow} 1$, therefore $\det_{\mathcal{N}(G)}(f) = 1$, and this completes the proof.
\end{proof}

\subsection{The $L^2$-Alexander invariant}

Let $K \subset S^3$ be a knot, $G_K$ its knot group, and $P = \big\langle g_1, \ldots, g_k \, \big| \, r_1, \ldots r_{k-1} \big\rangle$ a Wirtinger presentation of $G_K$.

For $t \in \C ^*$ we define the algebra homomorphism:

\begin{equation*}
    \psi_{K,t} \colon\left(
    \begin{aligned}
        \C[G_K] &\longrightarrow  \C[G_K] \\
        \sum_{g \in G_K} c_g \cdot [g] 
        &\longmapsto 
        \sum_{g \in G_K} c_g \cdot t^{\alpha_{K}(g)} \cdot [g]
    \end{aligned}
    \right)
\end{equation*}
and we also note $\psi_{K,t}$ its induction to any matrix ring with coefficients in $\C[G_K]$.
Think of it as a way of  \guillemotleft  tensoring by the abelianization representation\guillemotright .

We say that $(P,t)$ has \textit{Property $\mathcal{I}$} if 
$ R_{\psi_{K,t} (F_{P,1})}\colon l^2(G_K)^{k-1} \to l^2(G_K)^{k-1}$ is injective.

\begin{defn}
Let $K$ be a knot, let $P$ be a Wirtinger presentation of its knot group $G_K$, and let $t \in \C^{*}$.

If $(P,t)$ has Property $\mathcal{I}$ then \textit{the $L^2$-Alexander invariant of $K$ for the presentation $P$ at $t$} is written $\Delta^{(2)}_{K,P}(t)$ and is defined by:
\[
\Delta^{(2)}_{K,P}(t):= {\det}_{\mathcal{N}(G_K)} \left(R_{\psi_{K,t} (F_{P,1})} \right) \in [0,\infty[.
\]
\end{defn}

\begin{prop} \label{prop inv HT L2 alex}
Let $P$ and $Q$ be two Wirtinger presentations with deficiency one of the same knot group $G_K$, and let $X_P \subset \C^*$ (resp. $X_Q$) be the set of $t$ such that $(P,t)$ (resp. $(Q,t)$) has Property $\mathcal{I}$.

Then $X_P=X_Q$ and there is an integer $m$ such that $\Delta^{(2)}_{K,Q}(t) = \Delta^{(2)}_{K,P}(t) \cdot  |t|^m$ for all $t$ in $X_P$.
\end{prop}

The proof of this proposition is somewhat technical. It is based on a study of Tietze transformations between Wirtinger presentations and of how the respective associated operators are consequently modified by these transformations. Compare with \cite[Section 5]{Wada} and \cite[Proposition 3.4]{LZ06a}. We include the following detailed proof for the sake of completeness.

\begin{proof}
Let $P$ and $Q$ be two Wirtinger presentations with deficiency one of the same knot group $G_K$. This means that $P$ and $Q$ were constructed respectively from two diagrams $D$ and $D'$ of the same knot $K$. Therefore $D'$ is obtained from $D$ by a finite sequence of planar isotopies and Reidemeister moves. As explained in \cite[Lemma 6]{Wada}, this means that $Q$ can be obtained from $P$ by a finite sequence of certain Tietze transformations (and their inverses), that are:
\begin{itemize}
\item $I_a$. To replace one of the relators $r_i$ by its inverse $r_i^{-1}$
\item $I_b$. To replace one of the relators $r_i$ by its conjugate $w r_i w^{-1}$ where $w$ is a word in the generators.
\item $I_c$. To replace one of the relators $r_i$ by its product $r_i r_k$ with a different relator ($k \neq i$)
\item $II_W$. To add a new generator $x$ and a new relator $x=w$ where $w$ is of the form $x_j x_i x_j^{-1}$ or $x_j^{-1} x_i x_j$ with $x_i$ and $x_j$ some previous generators.
\item $III$. To apply a permutation on the generators.
\end{itemize}
Note that we specified a new transformation $III$ to describe the ambiguity in ordering the generators during the Wirtinger process, and that we only use $II_W$ and not the $II$ of \cite[Section 1]{Wada} because it is sufficient to describe the modifications caused by Reidemeister moves and it helps us ensure the following fact:
if a sequence of such Tietze moves transforms the Wirtinger presentation $P$ into the Wirtinger presentation $Q$, then all intermediate presentations are not necessarily of the Wirtinger form but they all have the fundamental property that \textit{their generators are all conjugates of one another}.

To establish the proposition, it suffices to prove that if $Q$ is obtained from $P$ by a single previous transformation, then $X_P = X_Q$ and there is an integer $m$ such that ${\Delta^{(2)}_{K,Q}(t) = \Delta^{(2)}_{K,P}(t) \cdot  |t|^m}$ for all $t$ in $X_P$.

Firstly, if $Q$ is obtained from $P$ by a $I_a$ move, for example the $j$-th relator $r$ is changed to $r^{-1}$, then by construction the respective free groups in the generators and quotient maps are the same (notably $Gr(P) = Gr(Q)$). 
Remark that this will also be the case for moves of type $I_b$, $I_c$ and $III$.
Since 
$$\overline{ \dfrac{\partial}{\partial x}(r^{-1}) }
=  \overline{\left ( - r^{-1} \dfrac{\partial}{\partial x}(r) \right )}
= - \overline{ \dfrac{\partial}{\partial x}(r) },
$$ we deduce that $\psi_{K,t}(F_{Q,1})$ is simply $\psi_{K,t}(F_{P,1})$ with the $j$-th column multiplied by $-1$. Therefore the right-multiplication associated operators are both injective for the same values of $t$, i.e. $X_P = X_Q$, and furthermore $\Delta^{(2)}_{K,Q}(t) = \Delta^{(2)}_{K,P}(t)$, by Proposition \ref{prop operations det} (2) and (4).

Secondly, if $Q$ is obtained from $P$ by a $I_b$ move, for example the $j$-th relator $r$ is changed to $w r w^{-1}$ with $w$ a word in the generators, then since 
\begin{align*}
  \overline{ \dfrac{\partial}{\partial x}(w r w^{-1}) }
&=   \overline{ \dfrac{\partial}{\partial x}(w) + w \dfrac{\partial}{\partial x}(r) + wr\dfrac{\partial}{\partial x}(w^{-1}) } \\
  &=  \overline{  \dfrac{\partial}{\partial x}(w) + w \dfrac{\partial}{\partial x}(r) + wr\left (-w^{-1} \dfrac{\partial}{\partial x}(w) \right ) } \\
  &=  \overline{w} \cdot  \overline{ \dfrac{\partial}{\partial x}(r) }
, 
\end{align*}
 we deduce that $\psi_{K,t}(F_{Q,1})$ is simply $\psi_{K,t}(F_{P,1})$ with the $j$-th column multiplied on the left by $\psi_{K,t}(\overline{w}) = t^m \overline{w} $ where $m$ is an integer. Thus the associated right-multiplication  operators are equal up to composition by the diagonal operator of $j$-th coefficient $R_{t^m w}$  and other coefficients $Id$; this operator is invertible and of Fuglede-Kadison determinant $|t|^m$. Therefore $\psi_{K,t}(F_{Q,1})$ and $\psi_{K,t}(F_{P,1})$ 
 are both injective for the same values of $t$, i.e. $X_P = X_Q$, and furthermore $\Delta^{(2)}_{K,Q}(t) = \Delta^{(2)}_{K,P}(t) \cdot  |t|^{n}$, by Proposition \ref{prop operations det} (2) and (4).

Thirdly, if $Q$ is obtained from $P$ by a $I_c$ move, for example the $j$-th relator $r$ is changed to $r r'$ with $r'$ the $l$-th relator, then since 
$$\overline{\dfrac{\partial}{\partial x}(r r') }
= \overline{  \dfrac{\partial}{\partial x}(r) + r \dfrac{\partial}{\partial x}(r')} 
= \overline{ \dfrac{\partial}{\partial x}(r) } + \overline{\dfrac{\partial}{\partial x}(r') }
,$$
we deduce that $\psi_{K,t}(F_{Q,1})$ is simply $\psi_{K,t}(F_{P,1})$ where the $l$-th column was added to the $j$-th one. Proposition  \ref{prop operations det} (2) and (4) let us conclude that $X_P = X_Q$ (composing by an invertible transvection operator does not change the injectivity) and that $\Delta^{(2)}_{K,Q}(t) = \Delta^{(2)}_{K,P}(t)$ (since a transvection operator has Fuglede-Kadison determinant $1$).

Fourthly, suppose that $Q$ is obtained from $P$ by a $II_W$ move, then write $P = \big\langle g_1, \ldots, g_k \, \big| \, r_1, \ldots r_{k-1} \big\rangle$ and $Q = \big\langle g_1, \ldots, g_k, h \, \big| \, r_1, \ldots r_{k-1}, w h^{-1} \big\rangle$  where $w$ is a word in the $g_i$.
Here $Gr(P)$ and $Gr(Q)$ are naturally isomorphic via
$$ Gr(P) = \F[g_i]/\langle r_j \rangle \hookrightarrow \F[g_i,h] / \langle r_j \rangle \twoheadrightarrow \F[g_i,h] / \langle r_j, wh^{-1} \rangle = Gr(Q)$$
(where $\langle r_j \rangle$ is the normal generated subgroup), therefore we dare an abuse of notation by writing
$$F_Q=\kbordermatrix{
 & r_1 & \ldots & r_{k-1} & wh^{-1} \\
g_1 & & & & *\\
\vdots & &  F_P & & \vdots \\
g_{k} & & & & * \\
h & 0 & \ldots & 0& -1 
}
$$
where the $*$ are elements of $\Z[G_K]$.
Thus $R_{\psi_{K,t} \left ( F_{Q,1}\right )}$ is injective if and only if $R_{\psi_{K,t} \left ( F_{P,1}\right )}$ is injective, i.e. $X_P=X_Q$. Hence, by Proposition \ref{prop operations det} (2) and (5),
${\Delta^{(2)}_{K,Q}(t) = \Delta^{(2)}_{K,P}(t)}$ for all $t\in X_P$.

Finally, suppose that $Q$ is obtained from $P$ by a $III$ move.  A permutation is a finite product of transpositions, therefore we can assume that the $III$ move is a transposition $\tau$.

Let us assume that $\tau$ leaves the first generator fixed. In this case the Fox matrix $F_{Q,1}$ is $F_{P,1}$ with two of its rows swapped, i.e. $F_{Q,1}$ is equal to $F_{P,1}$ multiplied by a permutation matrix $S$. Since the associated operator ${R_{\psi_{K,t}(S)} = R_S}$ is unitary, it is invertible and has Fuglede-Kadison determinant $1$. Thus $R_{\psi_{K,t} \left ( F_{Q,1}\right )}$ is injective if and only if $R_{\psi_{K,t} \left ( F_{P,1}\right )}$ is injective, i.e. $X_P=X_Q$. Hence, by Proposition \ref{prop operations det} (5), ${\Delta^{(2)}_{K,Q}(t) = \Delta^{(2)}_{K,P}(t)}$ for all $t\in X_P$.

Now let us assume that $\tau$ swaps the first and second generators. We write $F_P = \begin{pmatrix} L_1 \\ L_2 \\ \vdots \\ L_k \end{pmatrix}$ and
$F_Q = \begin{pmatrix} L_2 \\ L_1 \\ \vdots \\ L_k \end{pmatrix}$ where  $L_i = \left (\overline{  \dfrac{\partial r_j}{\partial g_i} }  \right )_{1 \leqslant j \leqslant k} $ denotes the $i$-th row of $F_P$.  Let us remind the reader that the generators $g_i$ are conjugates of one another, therefore they have the same image $1$ by the abelianization $\alpha_K$, which means that $\psi_{K,t}(g_i) = t g_i -1 $ for each $i$.

The fundamental formula of Fox calculus (see for instance \cite[Proposition 9.8]{BZ}) states that the following formula stands in $\C[G_K]$:  \begin{equation}
\tag{*}
\sum_{i=1}^k L_i \cdot (\overline{g_i} - 1) = 0.
\end{equation}\\
Let $A = \begin{pmatrix} 
R_{t g_2 - 1} & & & 0\\
 & R_{t g_3 - 1} & & \\
& & \ddots & \\
 0& & & R_{t g_k - 1} 
 \end{pmatrix} $,
$B = \begin{pmatrix} 
R_{t g_1 - 1} & & & 0\\
 & R_{t g_3 - 1} & & \\
& & \ddots & \\
 0& & & R_{t g_k - 1} 
 \end{pmatrix} $ and
$C = \begin{pmatrix} 
-Id & -Id & \ldots & -Id\\
 & Id & & 0 \\
& & \ddots & \\
 0& & & Id
 \end{pmatrix} $. We recognize a transvection matrix in $C$, which is thus invertible and with determinant $1$.
Proposition \ref{prop id - tg} tells us that $A$ and $B$ are injective and that their Fuglede-Kadison determinant is $\max(1,|t|)^{k-1}$.

The formula $(*)$ implies the following equality for operators:
\begin{align*}
C \circ A \circ R_{\psi_{K,t}(F_{P,1})} &=
\begin{pmatrix} -R_{\psi_{K,t} \left (L_1(g_1-1)\right )}
- R_{\psi_{K,t} \left (L_3(g_3-1)\right )} - \ldots -R_{\psi_{K,t} \left (L_k(g_k-1)\right )}
 \\ R_{\psi_{K,t} \left (L_3(g_3-1)\right )} \\ \vdots \\ R_{\psi_{K,t} \left (L_k(g_k-1)\right )} \end{pmatrix} \\
 &= \begin{pmatrix}  R_{\psi_{K,t} \left (L_2(g_2-1)\right )}
 \\ R_{\psi_{K,t} \left (L_3(g_3-1)\right )} \\ \vdots \\ R_{\psi_{K,t} \left (L_k(g_k-1)\right )} \end{pmatrix} 
 = B \circ R_{\psi_{K,t}(F_{Q,1})}
\end{align*}
Since $C, A$ and $B$ are injective, $R_{\psi_{K,t}(F_{P,1})}$ is injective if and only if $R_{\psi_{K,t}(F_{Q,1})}$ is injective, i.e. $X_P = X_Q$. Finally, by Proposition \ref{prop operations det} (4) and the values of the determinants of $A, B, C$, we conclude that ${\Delta^{(2)}_{K,Q}(t) = \Delta^{(2)}_{K,P}(t)}$ for all $t\in X_P$.

Any permutation can be decomposed as a finite product of transpositions swapping the first and second elements and transpositions leaving the first element fixed. Therefore the case of the $III$ move is treated, and the proposition is proven.
\end{proof}

\begin{defn}
Let $K$ be a knot. 
Let $P$ be any Wirtinger presentation of its knot group $G_K$.
Let $X_K$ be the set of $t\in  \C^*$ such that $(P,t)$ has Property $\mathcal{I}$ (according to the previous proposition, this does not depend on $P$).
The \textit{$L^2$-Alexander invariant of $K$ at $t$} is written $\left (t \mapsto \Delta^{(2)}_{K}(t)\right )$ and is defined as the class of $\left (t \mapsto \Delta^{(2)}_{K,P}(t)\right )$ up to multiplication by $(t \mapsto |t|^{\Z})$ on the maps from $X_K$ to $\R_{\geqslant 0}$.

It is a knot invariant by the previous proposition.
\end{defn}

\begin{rem}
Until now we know of no knots $K$ such that $X_K \neq \C^*$. However we know that $X_K$ always contains at least the entire unit circle, thanks to Theorem \ref{thm L2 volume}.
\end{rem}

\begin{rem} \label{rem switch generators}
Let us remark that we can take $F_{P,i}$ for any $i \neq 1$ instead of $F_{P,1}$ in the definition of the invariant, since it simply corresponds to an other Wirtinger presentation where the generators are permuted.
\end{rem}

\begin{ex} \label{ex L2 unknot}
Let us compute the invariant for the trivial knot $O$.

\begin{figure}[!h]
\centering
\begin{tikzpicture} [every path/.style={string } , every node/.style={transform shape , knot crossing , inner sep=1.5 pt } ] 
\begin{scope}[xshift=-1cm,scale=1]
	\node[rotate=0] (l) at (-1,0) {};
	\node[rotate=0] (r) at (1, 0 ) { } ;
	
\draw (l.center) .. controls (l.16 north east) and (r.16 north west) .. (r) ;
\draw (r) .. controls (r.32 south east) and (r.32 north east) .. (r.center) ;
\draw (r.center) .. controls (r.16 south west) and (l.16 south east) .. (l) ;
\draw (l) .. controls (l.32 north west) and (l.32 south west) .. (l.center) ;
\end{scope}
\end{tikzpicture}
\caption{A diagram for the unknot} \label{figure doubly twisted}
\end{figure}
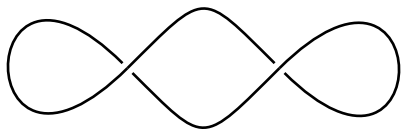

The  \guillemotleft  doubly twisted rubber band\guillemotright \ knot diagram of Figure \ref{figure doubly twisted} gives the Wirtinger presentation $P = \langle g, h | g h^{-1} \rangle$ of the unknot group $G_O$ (which is isomorphic to $\Z$), and the associated Fox matrix is $F_{P} = \begin{pmatrix} 1 \\ -1 \end{pmatrix}$.

Therefore for all $t$ in $\C^*$ , $R_{\psi_{O,t}(F_{P,1})} = -Id\colon l^2(G_O) \to l^2(G_O)$ has Property $\mathcal{I}$ and $\Delta_{O,P}^{(2)}(t)=1$.
Thus, the invariant for the trivial knot is the constant map equal to $1$.
\end{ex}

The following result is proven for the unit circle in \cite[Section 6]{LZ06a} and can be easily extended to $\C^*$.

\begin{prop}
(1) Let $K$ be a knot and $P$ a Wirtinger presentation of $G_K$, and let $t \in \C^*$. Then $(P,t)$ has Property $\mathcal{I}$ if and only if $(P,|t|)$ has Property $\mathcal{I}$.

(2) Let $K$ be a knot and $t \in \C ^*$, such that there is a Wirtinger presentation $P$ with $(P,t)$ having Property $\mathcal{I}$. Then $\Delta_K^{(2)}(t) = \Delta_K^{(2)}(|t|)$.
\end{prop}

We will now always assume $t>0$. The $L^2$-Alexander invariant is thus a class of maps from $\R_{>0}$ to $\R_{\geqslant 0}$ (up to multiplication by $(t \mapsto t^m), m \in \Z$).

The following theorem was proven by Lück for the $L^2$-torsion, but, similarly to Milnor's proof that the Alexander polynomial can be seen as a Reidemeister torsion, we can express the $L^2$-Alexander invariant of $K$ as a simple function of a $L^2$-torsion of $M_K$ (see for example \cite[Section 5]{LZ06a}).

\begin{thm}[\cite{Luc02}, Theorem 4.6] \label{thm L2 volume}
If $K$ is a non-trivial knot then the $3$-manifold $M_K$ is irreducible and, according to the JSJ-decomposition, splits along disjoint incompressible tori into pieces that are Seifert manifolds or hyperbolic manifolds. 
The hyperbolic pieces $M_1 , \ldots , M_h$ have all finite hyperbolic volume, and $$
\Delta_K^{(2)}(1) = \exp \left ({\frac{1}{6\pi}\sum_{i=1}^{h}vol(M_i)}\right )
= \exp \left ({\frac{1}{6\pi} \| M_K \|}\right )
$$
where $vol$ is the hyperbolic volume and $\|  . \|$  is the Gromov norm.
\end{thm}

Hence we now have the value of the invariant on one point. It is $1$ for torus knots, and for $K$ a hyperbolic knot it is an exponential of its hyperbolic volume, which was already known to be a strong knot invariant. We hope that the values for $t \neq 1$ can give relevant additional information.

To conclude this section, let us mention that we do not need to use a Wirtinger presentation $P$ to compute $\Delta_{K}^{(2)}(t)$. 

\begin{thm}[\cite{DW}, Theorem 3.5 and Proposition 6.2] \label{thm pas wirtinger}
\ 

(1) Let $K$ be a knot, $G_K$ its knot group, and $P = \big\langle g_1, \ldots, g_k \, \big| \, r_1, \ldots r_{k-1} \big\rangle$ any deficiency one presentation of $G_K$. If $t >0$ is such that $(P,t)$ has Property $\mathcal{I}$, then 
$ \dfrac{\mathrm{det}_{\mathcal{N}(G_K)}(R_{\psi_{K,t}(F_{P,1})}) }
{ \max(1,t)^{| \alpha_{K} (g_1) | -1}  } $
does not depend on $P$, and is equal to $\Delta^{(2)}_{K,P}(t)$ when $P$ is Wirtinger.
Thus we will also call this quantity $\Delta^{(2)}_{K,P}(t)$.

(2) If $K$ is the $(p,q)$-torus knot, then for any $t>0$, 
$\Delta_K^{(2)}(t)$ is defined and equals $\max(1,t)^{(|p|-1)(|q|-1)}$.
\end{thm}

We will use this powerful result to prove the cabling formula in Section \ref{section cabling}.

\begin{rem}
This theorem implies that the $L^2$-Alexander invariant is not a complete knot invariant. For example $T({2,7})$ and $T({3,4})$ are distinct torus knots but they both have 
$t \mapsto \max(1,t)^6$
as their $L^2$-Alexander invariant.

However the $L^2$-Alexander invariant detects if a knot is the unknot, as we will see in Section \ref{section unknot}.
\end{rem}

We can also use this theorem to compute the invariant of the mirror image of a knot.

\begin{prop}
Let $K$ be a knot in $S^3$ and $K^*$ its mirror image. Let $P$ be a Wirtinger presentation of $G_K$ and let $t>0$. Suppose $(P,t)$ has Property $\mathcal{I}$.

Then $G_{K^*}$ admits a group presentation $P^*$ naturally obtained from $P$, $(P^*,t^{-1})$ has Property $\mathcal{I}$ and 
$\Delta_{K^*}^{(2)}(t^{-1}) = \Delta_K^{(2)}(t)$.
\end{prop}

\begin{proof}

Take a diagram $D$ of $K$ and its image $D'$ by a planar reflection by a line not intersecting $D$. Then $D'$ is a diagram for $K^*$.
Take a base point in $\R^3$ above the plane of the diagrams $D$ and $D'$.

Each crossing of $D$ corresponds to a crossing of $D'$ as in Figure \ref{figure mirror}.

\begin{figure}[!h] 
\centering
\begin{tikzpicture} [every path/.style={string } , every node/.style={transform shape , knot crossing , inner sep=1.5 pt } ] 

\begin{scope}[xshift=0cm,scale=1]
	\node[rotate=0] (h) at (0,2) {};
	\node[rotate=0] (g) at (-2,0) {} ;
	\node[rotate=0] (c) at (0,0) {};
	\node[rotate=0] (d) at (2,0) {};
	\node[rotate=0] (b) at (0,-2) {};

\draw (g.center) .. controls (g.16 east) and (c.16 west) .. (c) ;
\draw[->] (c) .. controls (c.16 east) and (d.16 west) .. (d.center) ;
\draw[->] (b.center) .. controls (b.16 north) and (h.16 south) .. (h) ;
\end{scope}

\begin{scope}[xshift=0cm,scale=1,color=blue]
	\node (ah) at (0,1) {};
		\node (ahd) at (0.7,1) {};
		\node (ahg) at (-0.7,1) {};
	
\draw (ahd.center) .. controls (ahd.16 west) and (ah.16 east) .. (ah) ;
\draw[->] (ah) .. controls (ah.16 west) and (ahg.16 east) .. (ahg.center) ;
\draw (ahg) node[above left]{$a$} ;

	\node (ab) at (0,-1) {};
		\node (abd) at (0.7,-1) {};
		\node (abg) at (-0.7,-1) {};
	
\draw (abd.center) .. controls (abd.16 west) and (ab.16 east) .. (ab) ;
\draw[->] (ab) .. controls (ab.16 west) and (abg.16 east) .. (abg.center) ;
\draw (abg) node[below left]{$a$} ;

	\node (bg) at (-1,0) {};
		\node (bgb) at (-1,-0.7) {};
		\node (bgh) at (-1,0.7) {};
	
\draw (bgb.center) .. controls (bgb.16 north) and (bg.16 south) .. (bg) ;
\draw[->] (bg) .. controls (bg.16 north) and (bgh.16 south) .. (bgh.center) ;
\draw (bgh) node[above left]{$b$} ;

	\node (cd) at (1,0) {};
		\node (cdb) at (1,-0.7) {};
		\node (cdh) at (1,0.7) {};
	
\draw (cdb.center) .. controls (cdb.16 north) and (cd.16 south) .. (cd) ;
\draw[->] (cd) .. controls (cd.16 north) and (cdh.16 south) .. (cdh.center) ;
\draw (cdh) node[above right]{$c$} ;

\end{scope}

\begin{scope}[xshift=6cm,scale=1]
	\node[rotate=0] (h) at (0,2) {};
	\node[rotate=0] (g) at (-2,0) {} ;
	\node[rotate=0] (c) at (0,0) {};
	\node[rotate=0] (d) at (2,0) {};
	\node[rotate=0] (b) at (0,-2) {};

\draw[<-] (g.center) .. controls (g.16 east) and (c.16 west) .. (c) ;
\draw (c) .. controls (c.16 east) and (d.16 west) .. (d.center) ;
\draw[->] (b.center) .. controls (b.16 north) and (h.16 south) .. (h) ;

\end{scope}

\begin{scope}[xshift=6cm,scale=1,color=blue]
	\node (ah) at (0,1) {};
		\node (ahd) at (0.7,1) {};
		\node (ahg) at (-0.7,1) {};
	
\draw (ahg.center) .. controls (ahg.16 east) and (ah.16 west) .. (ah) ;
\draw[->] (ah) .. controls (ah.16 east) and (ahd.16 west) .. (ahd.center) ;
\draw (ahd) node[above right]{$A$} ;

	\node (ab) at (0,-1) {};
		\node (abd) at (0.7,-1) {};
		\node (abg) at (-0.7,-1) {};
	
\draw (abg.center) .. controls (abg.16 east) and (ab.16 west) .. (ab) ;
\draw[->] (ab) .. controls (ab.16 east) and (abd.16 west) .. (abd.center) ;
\draw (abd) node[below right]{$A$} ;

	\node (bg) at (-1,0) {};
		\node (bgb) at (-1,-0.7) {};
		\node (bgh) at (-1,0.7) {};
	
\draw (bgb.center) .. controls (bgb.16 north) and (bg.16 south) .. (bg) ;
\draw[->] (bg) .. controls (bg.16 north) and (bgh.16 south) .. (bgh.center) ;
\draw (bgh) node[above left]{$C$} ;

	\node (cd) at (1,0) {};
		\node (cdb) at (1,-0.7) {};
		\node (cdh) at (1,0.7) {};
	
\draw (cdb.center) .. controls (cdb.16 north) and (cd.16 south) .. (cd) ;
\draw[->] (cd) .. controls (cd.16 north) and (cdh.16 south) .. (cdh.center) ;
\draw (cdh) node[above right]{$B$} ;

\end{scope}

\end{tikzpicture}
\caption{A crossing of $D$, its mirror image in $D'$, and the associated meridian loops} \label{figure mirror}
\end{figure}
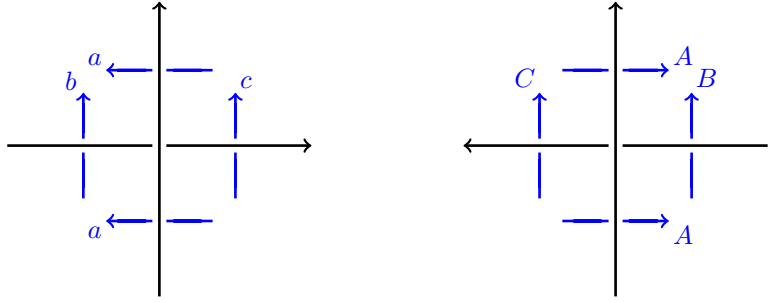

Let $P = \langle a_i | r_j \rangle$ be a Wirtinger presentation of $G_K = \pi_1(S^3 \setminus K)$ associated to $D$. Its relators are of the form $aba^{-1}c^{-1}$. As in Figure \ref{figure mirror}, for each generator $a_i$ of $P$, define $A_i$ a (negatively-oriented) meridian loop of $D'$, and for $r_j = aba^{-1}c^{-1}$, define $R_j = ABA^{-1}C^{-1}$. Then $P^* = \langle A_i | R_j \rangle$ is a presentation for $G_{K^*} = \pi_1(S^3 \setminus K^*)$. Note that $\alpha_{K^*}(A_i) = -1$ for all $i$.

Let $\phi : G_K \to G_{K^*}$ denote the natural group isomorphism sending $a_i$ to $A_i$ and its induction on the associated complex group algebras. Then 
$$\begin{matrix}
\C[G_K] & \overset{\psi_{K,t}}{\longrightarrow} & \C[G_K] \\
\downarrow \phi & & \downarrow \phi \\
\C[G_{K^*}] & \overset{\psi_{K^*,t^{-1}}}{\longrightarrow} & \C[G_{K^*}]
\end{matrix}
$$
is a commutative diagram, since $\psi_{K^*,t^{-1}}(A_i) = t A_i$ for all $i$.

Suppose $(P,t)$ has Property $\mathcal{I}$, thus $R_{\psi_{K,t}(F_{P,1})}$ is injective. Therefore, by Proposition \ref{prop i*} (1), Lemma \ref{lem psi i*} and Proposition \ref{prop i*} (2), in this order,
$$(\phi)_*( R_{\psi_{K,t}(F_{P,1})} )
= R_{{\phi}(\psi_{K,t}(F_{P,1}))}
= R_{\psi_{K^*,t^{-1}}({\phi}(F_{P,1}))}
= R_{\psi_{K^*,t^{-1}}(F_{P^*,1})}$$
is injective. Thus $(P^*, t^{-1})$ has Property $\mathcal{I}$.

By Theorem \ref{thm pas wirtinger}, since $P^*$ has deficiency one, 
$$\Delta_{K^*}^{(2)}(t^{-1}) = 
\dfrac{\mathrm{det}_{\mathcal{N}(G_{K^*})}(R_{\psi_{K^*,t^{-1}}(F_{P^*,1})}) }
{ \max(1,t)^{| \alpha_{K^*} (A_1) | -1}  } 
= 
\mathrm{det}_{\mathcal{N}(G_{K^*})}\left ((\phi)_*( R_{\psi_{K,t}(F_{P,1})} )\right ),
$$
and by Proposition \ref{prop operations det} (6) we conclude that
$\Delta_{K^*}^{(2)}(t^{-1}) = \Delta_K^{(2)}(t)$.
\end{proof}

\section{The $L^2$-Alexander invariant of a composite knot} \label{section composite}

Let $K_1$ and $K_2$ be knots in $S^3$ and $K$ their connected sum. We prove that the $L^2$-Alexander invariant of $K$ can be calculated from those of its factors. This multiplicativity of the invariant can be compared to the classical property of the Alexander polynomial of a composite knot, cf for example \cite[Proposition 8.14]{BZ}.

\begin{lem} \label{lem psi i*}
Let $K$ be the connected sum of $K_1$ and $K_2$, with $G, G_1$ and $G_2$ their respective groups.

Then for $j=1,2$ and for all $t>0$ we have the commutative diagram
$$\begin{matrix}
\C[G_j] & \overset{\psi_{K_j,t}}{\longrightarrow} & \C[G_j] \\
\downarrow i_j & & \downarrow i_j \\
\C[G] & \overset{\psi_{K,t}}{\longrightarrow} & \C[G]
\end{matrix}
$$
where $i_j\colon G_j \hookrightarrow G$ denotes both the group inclusion of Proposition \ref{prop somme inj} and its induction on the complex group algebras.
\end{lem}

\begin{proof}
Let us take $P_1$, $P_2$ and $P$ like in Proposition \ref{prop Amalgame}, and $ t > 0$. We have 
$$P_1 = \langle x_1 , \ldots , x_k | r_1 , \ldots, r_{k-1}   \rangle,$$
$$P_2 = \langle y_1 , \ldots , y_l | s_1 , \ldots, s_{l-1}   \rangle,$$ 
$$P = \langle x_1 , \ldots , x_k, y_1 , \ldots , y_l | r_1 , \ldots, r_{k-1}, s_1, \ldots, s_{l-1}, x_k y_l^{-1}   \rangle.$$

These three presentations are Wirtinger, therefore the $x_i$ are sent to $1$ by $\alpha_{K_1}$ as elements of $G_1$ and by $\alpha_K$ as elements of $G$, and the same can be said for the generators $y_j$.

Therefore the diagram is commutative for any $[g] \in \C[G_j]$ where $g$ is a generator of $P_1$ or $P_2$. The result follows from the fact that the $\psi_{.,t}$ and $i_j$ are algebra homomorphisms and that the previous $[g]$ generate the two group algebras.
\end{proof}

\begin{thm} \label{L2 somme}
Let $K$ be the connected sum of $K_1$ and $K_2$, with $G, G_1$ and $G_2$ their respective groups, and $P, P_1, P_2$ the presentations given by Proposition \ref{prop Amalgame}.

Let $t$ be any positive number. If we assume that $(P_1,t)$ and $(P_2,t)$ have Property $\mathcal{I}$,
then $(P,t)$ has Property $\mathcal{I}$ and
$\Delta^{(2)}_K(t) = \Delta^{(2)}_{K_1}(t) \Delta^{(2)}_{K_2}(t)$.

\end{thm}

\begin{proof}
Let $P_1$, $P_2$ and $P$ be like in Proposition \ref{prop Amalgame}, and $ t > 0$. 
We have two injective group homomorphisms $i_1 \colon G_1 \hookrightarrow G$ and $i_2 \colon G_2 \hookrightarrow G$ by Proposition \ref{prop somme inj}.

The values of $P, P_1, P_2$ imply that $R_{\psi_{K,t}(F_P)}$ is written:

$\kbordermatrix{
 & r_1 & \ldots & r_{k-1} & \omit\vrule height 1ex & s_1 & \ldots & s_{l-1} & \omit\vrule height 1ex & x_k y_l^{-1} \\
x_1 & & & & \omit\vrule & 0 & \ldots & 0 & \omit\vrule & 0 \\
\vdots & & R_{\psi_{K,t}({i_1}(F_{P_1,k}))} & & \omit\vrule & \vdots & & \vdots & \omit\vrule & \vdots \\
x_{k-1} & & & & \omit\vrule & 0 & \ldots & 0 & \omit\vrule & 0 \\
\cline{2-10}
x_k & & * & & \omit\vrule & 0 & \ldots & 0 & \omit\vrule & Id \\
\cline{2-10}
y_1 &0 &\ldots &0 & \omit\vrule & &  & & \omit\vrule & 0 \\
\vdots & \vdots & & \vdots & \omit\vrule &  &  R_{\psi_{K,t}({i_2}(F_{P_2,l}))} &  & \omit\vrule & \vdots \\
y_{l-1} & 0 & \ldots & 0 & \omit\vrule & & & & \omit\vrule & 0 \\
\cline{2-10}
y_l & 0 & \ldots & 0 & \omit\vrule &  &  * &  & \omit\vrule & -Id 
}$

$(P_1,t)$ has Property $\mathcal{I}$ thus $R_{\psi_{K_1,t}(F_{P_1,k})}$ is injective (by Remark \ref{rem switch generators}). Therefore, by Proposition \ref{prop i*} (1), Lemma \ref{lem psi i*} and Proposition \ref{prop i*} (2), in this order,
$$(i_1)_*( R_{\psi_{K_1,t}(F_{P_1,k})} )
= R_{{i_1}(\psi_{K_1,t}(F_{P_1,k}))}
= R_{\psi_{K,t}({i_1}(F_{P_1,k}))}$$
is injective.
Similarly, $R_{\psi_{K,t}({i_2}(F_{P_2,l}))}$
is injective. 
Finally, $ - Id_{\ell^2(G)}$ is clearly injective.

Therefore the block trigonal matrix $R_{\psi_{K,t}(F_{P,k})}$ is injective, thus, by Remark \ref{rem switch generators}, $(P,t)$ has Property $\mathcal{I}$.

Hence by Proposition \ref{prop operations det} (5) and (2),
$$
det_{\mathcal{N}(G)}\left (R_{\psi_{K,t}(F_{P,k})}\right ) =
det_{\mathcal{N}(G)}\left (R_{\psi_{K,t}({i_1}(F_{P_1,k}))}\right ) \cdot 
det_{\mathcal{N}(G)}\left (R_{\psi_{K,t}({i_2}(F_{P_2,l}))}\right ).
$$

Finally, 
$$det_{\mathcal{N}(G)}\left (R_{\psi_{K,t}({i_1}(F_{P_1,k}))}\right ) =
det_{\mathcal{N}(G)}\left ( (i_1)_*( R_{\psi_{K_1,t}(F_{P_1,k})} ) \right )
= det_{\mathcal{N}(G_1)}\left (  R_{\psi_{K_1,t}(F_{P_1,k})}  \right )
$$ by Lemma \ref{lem psi i*} and Proposition \ref{prop operations det} (6). We use a similar argument for the second term, and thus 
$$\Delta^{(2)}_K(t) = \Delta^{(2)}_{K_1}(t) \Delta^{(2)}_{K_2}(t).$$

\end{proof}

\section{The $L^2$-Alexander invariant of a cable knot} \label{section cabling}

\begin{lem} \label{lem sat psi i*}
Let $S$ be the $(p,q)$-cable of $C$, and let $G_S, G_C$ be their respective groups.
Then for all $t>0$ we have the commutative diagram
$$\begin{matrix}
\C[G_C] & \overset{\psi_{C,t^p}}{\longrightarrow} & \C[G_C] \\
\downarrow i_C & & \downarrow i_C \\
\C[G_S] & \overset{\psi_{S,t}}{\longrightarrow} & \C[G_S]
\end{matrix}
$$
where $i_C \colon G_C \hookrightarrow G_S$ denotes both the group inclusion of Proposition \ref{prop sat inj} and its induction on the complex group algebras.
\end{lem}

\begin{proof}

Let us take $P_C = \langle a_1 , \ldots , a_k | r_1 , \ldots, r_{k-1} \rangle $  and
$$P_S = \langle a_1 , \ldots , a_k, x, \lambda | r_1 , \ldots, r_{k-1}, x^p  a_k^{-q} \lambda^{-p} ,
\lambda^{-1} W(a_i) 
   \rangle$$
 like in Proposition \ref{prop groupe cable}. Let $ t > 0$. 

Proposition \ref{prop groupe cable} (2) tells us that every $a_i$ is sent to $1$ by $\alpha_C$ as an element of $G_C$ and is sent to $p$ by $\alpha_S$ as an element of $G_S$.

Therefore the diagram is commutative for any $[a_i] \in \C[G_C]$ where $a_i$ is a generator of $P_C$. The lemma follows from the fact that $\psi_{C,t^p}$, $\psi_{S,t}$ and $i_C$ are algebra homomorphisms and that the $[a_i]$ generate $\C[G_C]$.
\end{proof}

\begin{lem}  \label{lem 1+t+t2}
Let $G$ be a discrete countable group, let $g \in G$ of infinite order, let $p$ be a positive integer and let $t>0$. 
Then $Id + t R_g + \ldots + t^{(p-1)}  R_{g^{p-1}}$ is injective and
$$det_{\mathcal{N}(G)}\left (Id + t R_g + \ldots + t^{(p-1)}  R_{g^{p-1}}\right ) = \max ( 1,t)^{p-1}.$$
\end{lem}

\begin{proof}
Let us note $R = Id + t R_g + \ldots + t^{(p-1)}  R_{g^{p-1}}$.

We have $ (Id - tR_g) \circ R = Id - t^p R_{g^p}$. By Proposition \ref{prop id - tg},  $Id - t^p R_{g^p}$ is injective, therefore $R$ is injective.

Both $Id -t R_g$ and $R$ are injective, therefore, by Proposition \ref{prop operations det} (4),
$$det_{\mathcal{N}(G)}\left (Id - t^p R_{g^p}\right )
=det_{\mathcal{N}(G)}\left (Id -t R_g\right ) \cdot   det_{\mathcal{N}(G)}\left (R\right ).$$

Thus, by Proposition \ref{prop id - tg}, 
$ \max(1,t^p) = \max(1,t) \cdot  \det_{\mathcal{N}(G)}\left (R\right )$ and the lemma follows.

\end{proof}

\begin{thm} \label{L2 cable}
Let $S$ be the $(p,q)$-cable knot of companion knot $C$, $G_S, G_C$ their respective groups, and $t$ any positive real number. 

If there exists $P_w$ a Wirtinger presentation of $G_C$ such that $(P_w, t^p)$ has Property $\mathcal{I}$,
then there is a presentation $P_S$ of $G_S$ such that $(P_S,t)$ has Property $\mathcal{I}$, and
$$\Delta_S^{(2)}(t) = \Delta_C^{(2)}(t^p)  \cdot \max(1,t)^{(|p|-1) (|q|-1)}  = \Delta_C^{(2)}(t^p) \Delta_{T(p,q)}^{(2)}(t).$$
\end{thm}

\begin{proof}

Let $P_C = \langle a_1 , \ldots , a_k | r_1 , \ldots, r_{k-1} \rangle $  and
$$P_S = \langle a_1 , \ldots , a_k, x, \lambda | r_1 , \ldots, r_{k-1}, x^p  a_k^{-q} \lambda^{-p},
\lambda^{-1} W(a_i) 
   \rangle$$
 be like in Proposition \ref{prop groupe cable}. 

Remark that $P_C$ is a Wirtinger presentation of $G_C$, as is $P_w$, therefore $(P_C,t^p)$ also has Property $\mathcal{I}$, by Proposition \ref{prop inv HT L2 alex}.

 Besides, $P_S$ is a presentation of deficiency one, thus by Theorem \ref{thm pas wirtinger}, $\Delta_S^{(2)}(u)$ will be equal to $\Delta_{S,P_S}^{(2)}(u)$ for any $u>0$ such that $(P_S,u)$ has Property $\mathcal{I}$.

Recall from Proposition \ref{prop groupe cable} (2) that $\alpha_{S}(a_i) = p, \alpha_{S} (x) = q$ and $\alpha_{S} (\lambda) = 0$.

The values of $P_S$ and $P_C$ imply that $R_{\psi_{S,t}(F_{P_S})}$ is written:

$\kbordermatrix{
 & r_1 & \ldots & r_{k-1} & \omit\vrule height 1ex & x^p a_k^{-q} \lambda ^{-p} & \omit\vrule height 1ex & \lambda W(a_i)^{-1} \\
a_1 & & & & \omit\vrule & 0 & \omit\vrule & * \\
\vdots & & R_{\psi_{S,t}({i_C}(F_{P_C,k}))} & & \omit\vrule & \vdots & \omit\vrule & \vdots \\
a_{k-1} & & & & \omit\vrule & 0 & \omit\vrule & * \\
\cline{2-8}
a_k & * & \ldots & * & \omit\vrule & *  & \omit\vrule & * \\
\cline{2-8}
x & 0 & \ldots & 0 & \omit\vrule & T  & \omit\vrule & 0 \\
\cline{2-8}
\lambda & 0 & \ldots & 0 & \omit\vrule &  * & \omit\vrule & Id 
}$

where $T = Id + t^q R_x + \ldots + t^{q(p-1)} R_{x^{p-1}}$ if $p$ is positive, and 
$$T = -t^{-q}R_{x^{-1}} - \ldots - t^{-q|p|}R_{x^p}
= \left (-t^{-q|p|}R_{x^p}\right ) \circ ( Id + t^q R_x + \ldots + t^{q(|p|-1)} R_{x^{|p|-1}})$$
 if $p$ is negative. In both cases $T$ is injective, by Lemma \ref{lem 1+t+t2} and the fact that $\left (-t^{-q|p|}R_{x^p}\right ) $ is invertible.

We know $(P_C,t^p)$ has Property $\mathcal{I}$, thus $R_{\psi_{C,t^p}(F_{P_C,k})}$ is injective, by Remark \ref{rem switch generators}.
We have the injective group homomorphism $i_C \colon G_C \hookrightarrow G_S$ by Proposition \ref{prop sat inj}.
Therefore, by Proposition \ref{prop i*} (1), Lemma \ref{lem sat psi i*} and Proposition \ref{prop i*} (2), in this order,
$$(i_C)_*( R_{\psi_{C,t^p}(F_{P_C,k})} )
= R_{{i_C}(\psi_{C,t^p}(F_{P_C,k}))}
= R_{\psi_{S,t}({i_C}(F_{P_C,k}))}$$
is injective.

Finally $Id_{\ell^2(G)}$ is clearly injective.

Thus the block trigonal square matrix $R_{\psi_{S,t}(F_{P_S,k})}$ is injective, hence,by Remark \ref{rem switch generators}, $(P_S,t)$ has Property $\mathcal{I}$. 
Therefore, by Proposition \ref{prop operations det} (5) and (2),
$$
det_{\mathcal{N}(G_S)}\left (R_{\psi_{S,t}(F_{P_S,k})}\right ) =
det_{\mathcal{N}(G_S)}\left (R_{\psi_{S,t}({i_C}(F_{P_C,k}))}\right ) \cdot 
det_{\mathcal{N}(G_S)}\left (T\right ).
$$

However we have 
$$det_{\mathcal{N}(G_S)}\left (R_{\psi_{S,t}({i_C}(F_{P_C,k}))}\right ) =
det_{\mathcal{N}(G_S)}\left ( (i_C)_*( R_{\psi_{C,t^p}(F_{P_C,k})} ) \right )
= det_{\mathcal{N}(G_C)}\left (  R_{\psi_{C,t^p}(F_{P_C,k})}  \right )
$$ by Lemma \ref{lem sat psi i*} and Proposition \ref{prop operations det} (6).

Besides, from Lemma \ref{lem 1+t+t2}, we have
$$det_{\mathcal{N}(G_S)}\left (Id + t^q R_x + \ldots + t^{q(|p|-1)}R_{x^{|p|-1}}\right ) = \max(1,t^q)^{|p|-1},$$
therefore, by the fact that $det_{\mathcal{N}(G_S)}\left (-t^{-q|p|}R_{x^p}\right ) \in t^\Z$ and Proposition \ref{prop operations det} (4), $det_{\mathcal{N}(G_S)}(T)$ is equal to $\max(1,t^q)^{|p|-1}$ up to $t^\Z$.

Note that for $t>0$ and any integer $k$, $\max(1,t^k) = t^{\frac{k-|k|}{2}} \max(1,t)^{|k|}$, therefore $\max(1,t^q)^{|p|-1} = \max(1,t)^{|q|(|p|-1)}$ up to $t^\Z$.

Finally, Theorem \ref{thm pas wirtinger}  tells us that $$
\Delta^{(2)}_{S}(t)=
 \dfrac{det_{\mathcal{N}(G_S)}(R_{\psi_{S,t}(F_{P_S,k})}) }
{ \max(1,t)^{|\alpha_{S} (a_k) | -1}  } 
= \dfrac{ det_{\mathcal{N}(G_C)}\left (  R_{\psi_{C,t^p}(F_{P_C,k})}  \right )  \cdot \max(1,t)^{|q|(|p|-1)} }
{ \max(1,t)^{|p| -1}  }.$$

Thus we have proven the formula
$$\Delta^{(2)}_{S}(t)= \Delta_C^{(2)}(t^p) \cdot  \max(1,t)^{(|p|-1)(|q|-1)}.
$$

\end{proof}

\begin{rem}\label{rem sat formula false}
A crucial part of this proof is the fact that the presentation of the group of the pattern knot inside its solid torus was easy to compute and manipulate (cf Section \ref{section}).

A general satellite formula mirroring the classical one for the Alexander polynomial (cf for instance \cite[Proposition 8.23]{BZ})is plainly untrue if written as $\Delta_{S_{C,P}}^{(2)}(t) = \Delta_C^{(2)}(t^{n_P}) \Delta_P^{(2)}(t)$. 

Indeed, if $P$ is a Whitehead double pattern inside the solid torus $T_P$, i.e. if $S_{C,P}$ is a Whitehead double of $C$, and if $C$ is a non trivial knot of Gromov norm zero, then $n_P$ is zero, $P$ is trivial in $S^3$ and $\Delta_{S_{C,P}}^{(2)}$ would then be the constant map $(t \mapsto 1)$ according to the previous formula; but we are going to show in Theorem \ref{thm detection unknot} that it cannot be since $S_{C,P}$ is not the unknot. 
\end{rem}

\begin{cor}
Let $K$ be a knot, $-K$ its inverse knot, and $P$ and $P_-$ Wirtinger presentations of their respective groups.
Then for all positive real numbers $t$, $(P,t)$ has Property $\mathcal{I}$ if and only if $(P_-,t^{-1})$ has Property $\mathcal{I}$, and in this case $$\Delta_{-K}^{(2)}(t^{-1}) = \Delta_K^{(2)}(t).$$
\end{cor}

\begin{proof}
Remark that $-K$ is a $(-1,m)$-cable of $K$ with $m$ any integer, and apply Theorem \ref{L2 cable}.
\end{proof}

\section{Detection of the unknot} \label{section unknot}

In \cite{Luc02}, L\"uck (Theorem 4.7 (2)) proves that the pair composed of the $L^2$-torsion and the Alexander polynomial detects the unknot. We prove a similar result for the $L^2$-Alexander invariant:

\begin{thm} \label{thm detection unknot}
Let $K$ be a knot in $S^3$.
The $L^2$-Alexander invariant of $K$ is trivial, i.e.
 $\left (t \mapsto \Delta_K^{(2)}(t) \right ) = (t \mapsto 1)$, if and only if $K$ is the trivial knot.
\end{thm}

This seems to confirm that the $L^2$-Alexander invariant can be seen as a generalization of both the $L^2$-torsion (i.e. the Gromov norm) and the Alexander polynomial.

\begin{proof}

First, let $K_0$ be an arbitrary knot.
If the exterior of $K_0$ has hyperbolic pieces in its JSJ decomposition, then $\Delta_{K_0}^{(2)}(1) \neq 1$, by Theorem \ref{thm L2 volume}. Therefore, let us assume $\tilde{K}$ is a knot whose exterior does not have hyperbolic pieces and such that $\Delta_{\tilde{K}}^{(2)} = (t \mapsto 1)$. Let us prove that $\tilde{K}$ is the unknot.

Besides, \cite[Lemma 5.5]{Mur} tells us that if we call $\mathcal{K}$ the class of knots generated by the unknot, the connected sum operation, and all cabling operations (for all torus knot patterns), then $\tilde{K} \in \mathcal{K}$.

Let us prove that for all knots $K$ in the class $\mathcal{K}$, $\ \Delta_K^{(2)} = (t \mapsto \max(1,t)^{n_K})$ where $n_K$ is a nonnegative integer. 

From Example \ref{ex L2 unknot}, it is true for the unknot and $n_O = 0$. Secondly, if the property is true for $K_1$ and $K_2$ in $\mathcal{K}$, then, by Theorem \ref{L2 somme}, it is true for their connected sum $K_1 \sharp K_2$ and $n_{K_1 \sharp K_2} = n_{K_1} + n_{K_2}$. Finally, if the property is true for $C \in \mathcal{K}$ and $S$ is the $(p,q)$-cable of $C$, then, it is true for $S$ and
$n_S = |p| \cdot n_C + (|p|-1)(|q|-1)$, by Theorem \ref{L2 cable}.

Observe that $n_{K_1 \sharp K_2}=0$ if and only if $n_{K_1}=n_{K_2}=0$, and $n_S=0$ if and only if $n_C=0$ and $p= \pm 1$ (i.e. the cabling operation is trivial or the knot inversion).
Therefore, the subclass $\mathcal{K}'$ of knots $K'$ in $\mathcal{K}$ such that $n_{K'}=0$ is exactly the class generated by $O$, the connected sum, the trivial cabling operation and the reversing of the orientation of the knot. But this class is reduced to $O$. Therefore, for $K \in \mathcal{K}$, $n_K=0$ if and only if $K=O$.

Thus, if $\tilde{K}$ is a knot whose exterior does not have hyperbolic pieces and such that $\Delta_{\tilde{K}}^{(2)} = (t \mapsto 1)$, then $\tilde{K}$ is the unknot. The theorem follows.

\end{proof}

\section{Proof of Proposition \ref{prop groupe cable}} \label{section}

The object of this section is to prove the proposition

\begin{divers}
Let us consider the $(p,q)$-cable knot $S$ of companion $C$. 

(1) There exists $P_C = \langle a_1 , \ldots , a_k | r_1 , \ldots, r_{k-1}   \rangle$ a Wirtinger presentation of $G_C$  such that
$$P_S = \langle a_1 , \ldots , a_k, x, \lambda | r_1 , \ldots, r_{k-1}, x^p  a_k^{-q} \lambda^{-p} ,
\lambda^{-1} W(a_i) 
   \rangle$$
is a presentation of $G_S$, with  $x$ and $\lambda$ the homotopy classes of the core and a longitude of $T_C$, and $W(a_i)$ a word in the $a_1, \ldots ,a_k$.

(2) Furthermore, $\alpha_{S}(x) = q$, $\alpha_{S}(\lambda) =0$ and $\alpha_{S}(a_i) = p$, for $i=1, \ldots , k$.

\end{divers}

\subsection{Group of a torus knot pattern}

Let $T_{int}$ be an open solid torus and $T_{ext}$ an open tubular neighboorhood of $T_{int}$, thus a second solid torus.
We will draw the torus knot $K = T(p,q)$ on the boundary of $T_{int}$. Let us take $pt$ any point on $\partial T_{int} \smallsetminus K$. It will be the base point for all the following fundamental groups.
Figure \ref{fig torus inside torus} (where $p=3$ and $q=4$) should clarify the notations.

\begin{figure}[!h]
\centering

\begin{tikzpicture}[every path/.style={string ,black} , every node/.style={transform shape , knot crossing , inner sep=1.5 pt } ]

\begin{scope}[xshift=0cm,rotate=0,scale=1.2]

\draw (-3.45,1.25) node {$K$} ;

\draw (-2,-0.8) node {$T_{int}$} ;

\coordinate (ch) at (-2,-1) ;
\coordinate (cb) at (-2,-2) ;
\draw [style=dotted] (ch) ..controls +(0.5,0) and +(0.5,0).. (cb) ;
\draw (ch) ..controls +(-0.5,0) and +(-0.5,0).. (cb) ;

\coordinate (t1) at (-3,2) ;
\coordinate (t2) at (3,2) ;
\coordinate (t3) at (-3,1) ;
\coordinate (t4) at (3,1) ;
\coordinate (t5) at (-3,-1) ;
\coordinate (t6) at (3,-1) ;
\coordinate (t7) at (-3,-2) ;
\coordinate (t8) at (3,-2) ;
\draw (t1) -- (t2) ..controls +(2.5,0) and +(2.5,0).. (t8) -- (t7) ..controls +(-2.5,0) and +(-2.5,0).. (t1) ;
\draw (t3) -- (t4) ..controls +(1,0) and +(1,0).. (t6) -- (t5) ..controls +(-1,0) and +(-1,0).. (t3) ;

\coordinate (k1) at (-3,1.8) ;
\coordinate (k2) at (3,1.8) ;
\coordinate (k3) at (-3,1.2) ;
\coordinate (k4) at (3,1.2) ;
\coordinate (k5) at (-3,-1.2) ;
\coordinate (k6) at (3,-1.2) ;
\coordinate (k7) at (-3,-1.8) ;
\coordinate (k8) at (3,-1.8) ;
\coordinate (k9) at (-3,1.5) ;
\coordinate (k10) at (3,1.5) ;
\coordinate (k11) at (3,-1.5) ;
\coordinate (k12) at (-3,-1.5) ;
\draw (k2) ..controls +(2.3,0) and +(2.3,0).. (k8) -- (k7) ..controls +(-2.3,0) and +(-2.3,0).. (k1) ;
\draw (k4) ..controls +(1.3,0) and +(1.3,0).. (k6) -- (k5) ..controls +(-1.3,0) and +(-1.3,0).. (k3) ;
\draw [style=dotted] (k10) ..controls +(1.8,0) and +(1.8,0).. (k11) -- (k12) ..controls +(-1.8,0) and +(-1.8,0).. (k9) ;

\begin{scope}[xshift=-3cm]
\coordinate (kh) at (0,1.8) ;
\coordinate (km) at (0,1.5) ;
\coordinate (kb) at (0,1.2) ;
\coordinate (khi) at (0.5,2) ;
\coordinate (kbi) at (1,1) ;
\coordinate (kh') at (1.5,1.8) ;
\coordinate (km') at (1.5,1.5) ;
\coordinate (kb') at (1.5,1.2) ;
\draw  (kh) ..controls +(0.1,0) and +(-0.1,0).. (kb') ;
\draw [style=dotted] (km) ..controls +(0.1,0) and +(-0.1,0).. (khi) ;
\draw  (khi) ..controls +(0.1,0) and +(-0.1,0).. (kh') ;
\draw  (kb) ..controls +(0.1,0) and +(-0.1,0).. (kbi) ;
\draw [style=dotted] (kbi) ..controls +(0.1,0) and +(-0.1,0).. (km') ;
\end{scope}

\begin{scope}[xshift=-1.5cm]
\coordinate (kh) at (0,1.8) ;
\coordinate (km) at (0,1.5) ;
\coordinate (kb) at (0,1.2) ;
\coordinate (khi) at (0.5,2) ;
\coordinate (kbi) at (1,1) ;
\coordinate (kh') at (1.5,1.8) ;
\coordinate (km') at (1.5,1.5) ;
\coordinate (kb') at (1.5,1.2) ;
\draw  (kh) ..controls +(0.1,0) and +(-0.1,0).. (kb') ;
\draw [style=dotted] (km) ..controls +(0.1,0) and +(-0.1,0).. (khi) ;
\draw  (khi) ..controls +(0.1,0) and +(-0.1,0).. (kh') ;
\draw  (kb) ..controls +(0.1,0) and +(-0.1,0).. (kbi) ;
\draw [style=dotted] (kbi) ..controls +(0.1,0) and +(-0.1,0).. (km') ;
\end{scope}

\begin{scope}[xshift=0cm]
\coordinate (kh) at (0,1.8) ;
\coordinate (km) at (0,1.5) ;
\coordinate (kb) at (0,1.2) ;
\coordinate (khi) at (0.5,2) ;
\coordinate (kbi) at (1,1) ;
\coordinate (kh') at (1.5,1.8) ;
\coordinate (km') at (1.5,1.5) ;
\coordinate (kb') at (1.5,1.2) ;
\draw  (kh) ..controls +(0.1,0) and +(-0.1,0).. (kb') ;
\draw [style=dotted] (km) ..controls +(0.1,0) and +(-0.1,0).. (khi) ;
\draw  (khi) ..controls +(0.1,0) and +(-0.1,0).. (kh') ;
\draw  (kb) ..controls +(0.1,0) and +(-0.1,0).. (kbi) ;
\draw [style=dotted] (kbi) ..controls +(0.1,0) and +(-0.1,0).. (km') ;
\end{scope}

\begin{scope}[xshift=1.5cm]
\coordinate (kh) at (0,1.8) ;
\coordinate (km) at (0,1.5) ;
\coordinate (kb) at (0,1.2) ;
\coordinate (khi) at (0.5,2) ;
\coordinate (kbi) at (1,1) ;
\coordinate (kh') at (1.5,1.8) ;
\coordinate (km') at (1.5,1.5) ;
\coordinate (kb') at (1.5,1.2) ;
\draw  (kh) ..controls +(0.1,0) and +(-0.1,0).. (kb') ;
\draw [style=dotted] (km) ..controls +(0.1,0) and +(-0.1,0).. (khi) ;
\draw  (khi) ..controls +(0.1,0) and +(-0.1,0).. (kh') ;
\draw  (kb) ..controls +(0.1,0) and +(-0.1,0).. (kbi) ;
\draw [style=dotted] (kbi) ..controls +(0.1,0) and +(-0.1,0).. (km') ;
\end{scope}

\draw (1,0) arc (70:110:2.92) ;
\draw (1,0) arc (-70:-110:2.92) ;

\coordinate (th) at (0,3) ;
\coordinate (td) at (5.5,0) ;
\coordinate (tb) at (0,-3) ;
\coordinate (tg) at (-5.5,0) ;
\draw (th) ..controls +(5,0) and +(0,1)..
(td) ..controls +(0,-1) and +(5,0)..
(tb) ..controls +(-5,0) and +(0,-1)..
(tg) ..controls +(0,1) and +(-5,0)..
(th) ;

\coordinate (c) at (0,0.17) ;
\draw [style=dotted] (th) ..controls +(0.5,0) and +(0.5,0).. (c) ;
\draw (th) ..controls +(-0.5,0) and +(-0.5,0).. (c) ;
\draw (0,3.2) node {$T_{ext}$} ;

\end{scope}
\end{tikzpicture}

\caption{The inside and outside tori $T_{int}$ and $T_{ext}$ and the $(p,q)$-torus knot $K$} \label{fig torus inside torus}
\end{figure}
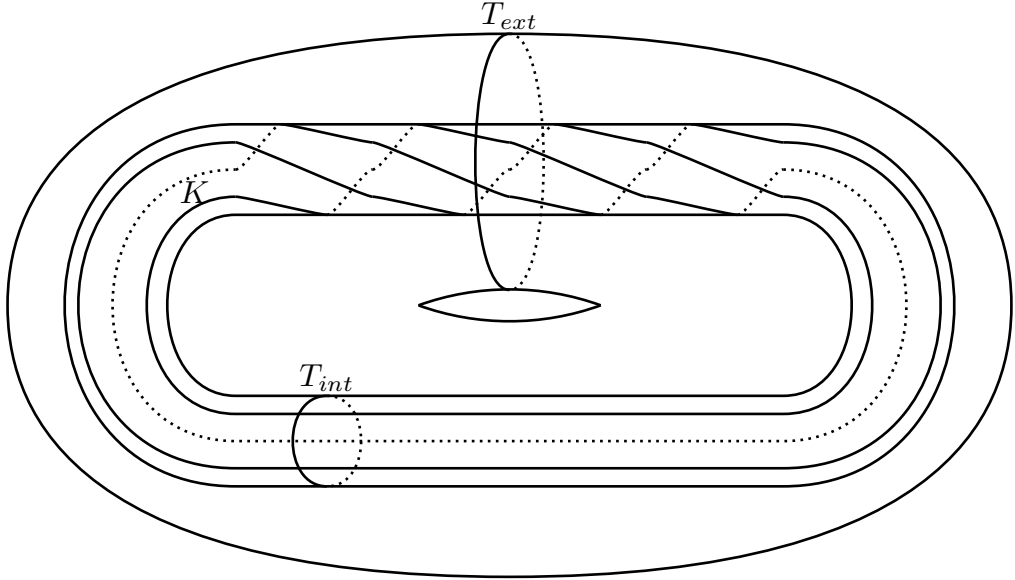

We want to prove the following result:

\begin{lem} \label{lem groupe Tpq}
$P_{p,q} = \langle x , y , \lambda | x^p = \lambda^p y^q, \lambda y = y \lambda \rangle$ is a presentation of \\
$\tilde{G}_{p,q} = \pi_1(T_{ext} \smallsetminus K)$.
Furthermore, the elements of $\tilde{G}_{p,q}$ represented by $\lambda$ and $y$ are the homotopy classes of a longitude curve and a meridian curve of $T_{ext} \setminus \overline{T_{int}}$, and $x$ is the homotopy class of the core of $T_{int}$.
\end{lem}

The following proof has been inspired by the calculation of the classical presentation of torus knot groups, cf for example \cite[Section 3.C]{Rol}.

\begin{proof}
We will use the Seifert-van Kampen theorem. 

We note 
$ U_1 = T_{ext} \smallsetminus ( T_{int} \sqcup K ) $, 
$U_2 = \overline{T_{int}} \smallsetminus K$, 
$W = T_{ext} \smallsetminus K$, $V = \partial T_{int} \smallsetminus K$ and $G_1, G_2, G, G_0$ their respective fundamental groups (for the same base point $pt$ in $V$).

$U_1$ can be deformed to $T_{ext} \smallsetminus T_{int}$ (by  \guillemotleft  filling up $K$\guillemotright ), and so it is homotopically equivalent to a $2$-torus. Thus $\langle y , \lambda | y \lambda = \lambda y \rangle $ is a presentation of $G_1$, where $y$ and $\lambda$ are the homotopy classes of a natural meridian-longitude system of $T_{ext} \setminus T_{int}$, see Figure \ref{fig meridian longitude}.

\begin{figure}[!h]
\centering

\begin{tikzpicture}[every path/.style={string ,black} , every node/.style={transform shape , knot crossing , inner sep=1.5 pt } ]

\begin{scope}[xshift=0cm,rotate=0,scale=1]

\draw (-2,-0.8) node {$T_{int}$} ;

\coordinate (ch) at (-2,-1) ;
\coordinate (cb) at (-2,-2) ;
\draw [style=dotted] (ch) ..controls +(0.5,0) and +(0.5,0).. (cb) ;
\draw (ch) ..controls +(-0.5,0) and +(-0.5,0).. (cb) ;

\coordinate (t1) at (-3,2) ;
\coordinate (t2) at (3,2) ;
\coordinate (t3) at (-3,1) ;
\coordinate (t4) at (3,1) ;
\coordinate (t5) at (-3,-1) ;
\coordinate (t6) at (3,-1) ;
\coordinate (t7) at (-3,-2) ;
\coordinate (t8) at (3,-2) ;
\draw (t1) -- (t2) ..controls +(2.5,0) and +(2.5,0).. (t8) -- (t7) ..controls +(-2.5,0) and +(-2.5,0).. (t1) ;
\draw (t3) -- (t4) ..controls +(1,0) and +(1,0).. (t6) -- (t5) ..controls +(-1,0) and +(-1,0).. (t3) ;

\coordinate (k1) at (-3,1.8) ;
\coordinate (k2) at (3,1.8) ;
\coordinate (k3) at (-3,1.2) ;
\coordinate (k4) at (3,1.2) ;
\coordinate (k5) at (-3,-1.2) ;
\coordinate (k6) at (3,-1.2) ;
\coordinate (k7) at (-3,-1.8) ;
\coordinate (k8) at (3,-1.8) ;
\coordinate (k9) at (-3,1.5) ;
\coordinate (k10) at (3,1.5) ;
\coordinate (k11) at (3,-1.5) ;
\coordinate (k12) at (-3,-1.5) ;
\draw (k2) ..controls +(2.3,0) and +(2.3,0).. (k8) -- (k7) ..controls +(-2.3,0) and +(-2.3,0).. (k1) ;
\draw (k4) ..controls +(1.3,0) and +(1.3,0).. (k6) -- (k5) ..controls +(-1.3,0) and +(-1.3,0).. (k3) ;
\draw [style=dotted] (k10) ..controls +(1.8,0) and +(1.8,0).. (k11) -- (k12) ..controls +(-1.8,0) and +(-1.8,0).. (k9) ;

\begin{scope}[xshift=-3cm]
\coordinate (kh) at (0,1.8) ;
\coordinate (km) at (0,1.5) ;
\coordinate (kb) at (0,1.2) ;
\coordinate (khi) at (0.5,2) ;
\coordinate (kbi) at (1,1) ;
\coordinate (kh') at (1.5,1.8) ;
\coordinate (km') at (1.5,1.5) ;
\coordinate (kb') at (1.5,1.2) ;
\draw  (kh) ..controls +(0.1,0) and +(-0.1,0).. (kb') ;
\draw [style=dotted] (km) ..controls +(0.1,0) and +(-0.1,0).. (khi) ;
\draw  (khi) ..controls +(0.1,0) and +(-0.1,0).. (kh') ;
\draw  (kb) ..controls +(0.1,0) and +(-0.1,0).. (kbi) ;
\draw [style=dotted] (kbi) ..controls +(0.1,0) and +(-0.1,0).. (km') ;
\end{scope}

\begin{scope}[xshift=-1.5cm]
\coordinate (kh) at (0,1.8) ;
\coordinate (km) at (0,1.5) ;
\coordinate (kb) at (0,1.2) ;
\coordinate (khi) at (0.5,2) ;
\coordinate (kbi) at (1,1) ;
\coordinate (kh') at (1.5,1.8) ;
\coordinate (km') at (1.5,1.5) ;
\coordinate (kb') at (1.5,1.2) ;
\draw  (kh) ..controls +(0.1,0) and +(-0.1,0).. (kb') ;
\draw [style=dotted] (km) ..controls +(0.1,0) and +(-0.1,0).. (khi) ;
\draw  (khi) ..controls +(0.1,0) and +(-0.1,0).. (kh') ;
\draw  (kb) ..controls +(0.1,0) and +(-0.1,0).. (kbi) ;
\draw [style=dotted] (kbi) ..controls +(0.1,0) and +(-0.1,0).. (km') ;
\end{scope}

\begin{scope}[xshift=0cm]
\coordinate (kh) at (0,1.8) ;
\coordinate (km) at (0,1.5) ;
\coordinate (kb) at (0,1.2) ;
\coordinate (khi) at (0.5,2) ;
\coordinate (kbi) at (1,1) ;
\coordinate (kh') at (1.5,1.8) ;
\coordinate (km') at (1.5,1.5) ;
\coordinate (kb') at (1.5,1.2) ;
\draw  (kh) ..controls +(0.1,0) and +(-0.1,0).. (kb') ;
\draw [style=dotted] (km) ..controls +(0.1,0) and +(-0.1,0).. (khi) ;
\draw  (khi) ..controls +(0.1,0) and +(-0.1,0).. (kh') ;
\draw  (kb) ..controls +(0.1,0) and +(-0.1,0).. (kbi) ;
\draw [style=dotted] (kbi) ..controls +(0.1,0) and +(-0.1,0).. (km') ;
\end{scope}

\begin{scope}[xshift=1.5cm]
\coordinate (kh) at (0,1.8) ;
\coordinate (km) at (0,1.5) ;
\coordinate (kb) at (0,1.2) ;
\coordinate (khi) at (0.5,2) ;
\coordinate (kbi) at (1,1) ;
\coordinate (kh') at (1.5,1.8) ;
\coordinate (km') at (1.5,1.5) ;
\coordinate (kb') at (1.5,1.2) ;
\draw  (kh) ..controls +(0.1,0) and +(-0.1,0).. (kb') ;
\draw [style=dotted] (km) ..controls +(0.1,0) and +(-0.1,0).. (khi) ;
\draw  (khi) ..controls +(0.1,0) and +(-0.1,0).. (kh') ;
\draw  (kb) ..controls +(0.1,0) and +(-0.1,0).. (kbi) ;
\draw [style=dotted] (kbi) ..controls +(0.1,0) and +(-0.1,0).. (km') ;
\end{scope}

\coordinate (yh) at (0,-0.7) ;
\coordinate (yb) at (0,-2.3) ;
\draw [<-] [very thick, color=brown] (0.3,-1) -- (0.3,-0.7) -- (0,-0.7);
\draw [very thick, color=brown] (0,-0.7) -- (-0.3,-0.7) -- (-0.3,-2.3) -- (0.3,-2.3) -- (0.3,-2) ;
\draw [very thick, color=brown] (-0.45,-0.75) node {$y$} ;

\draw (1,0) arc (70:110:2.92) ;
\draw (1,0) arc (-70:-110:2.92) ;

\coordinate (th) at (0,3) ;
\coordinate (td) at (5.5,0) ;
\coordinate (tb) at (0,-3) ;
\coordinate (tg) at (-5.5,0) ;
\draw (th) ..controls +(5,0) and +(0,1)..
(td) ..controls +(0,-1) and +(5,0)..
(tb) ..controls +(-5,0) and +(0,-1)..
(tg) ..controls +(0,1) and +(-5,0)..
(th) ;

\coordinate (c) at (0,0.17) ;
\draw [style=dotted] (th) ..controls +(0.5,0) and +(0.5,0).. (c) ;
\draw (th) ..controls +(-0.5,0) and +(-0.5,0).. (c) ;
\draw (0,3.2) node {$T_{ext}$} ;

\draw [->] [color=magenta, very thick] (2,0.4) -- (2,-0.4) -- (-1.5,-0.4) ;
\draw [color=magenta, very thick] (-1.5,-0.4) -- (-2,-0.4) -- (-2,0.4) -- (2,0.4) ;
\draw [color=magenta](-1.5,-0.15) node {$\lambda$} ;

\end{scope}
\end{tikzpicture}

\caption{A natural meridian-longitude system} \label{fig meridian longitude}
\end{figure}
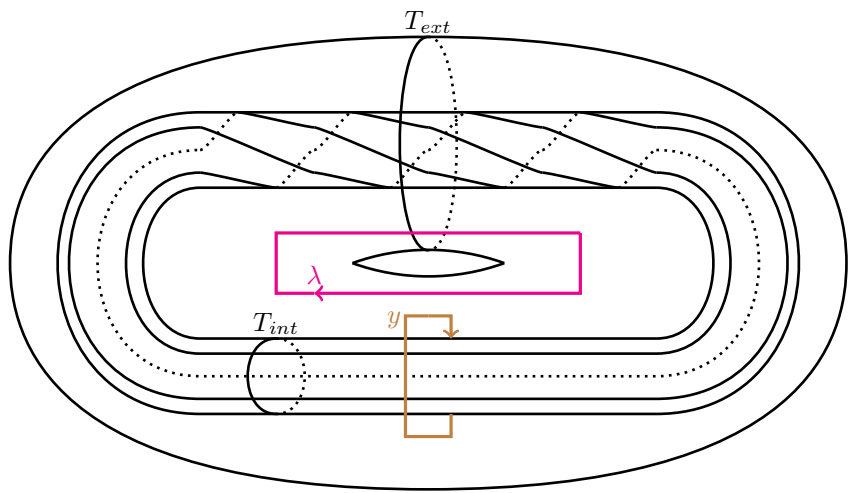

$U_2$ can be deformed to $T_{int}$ by a similar process, therefore $G_{2}$ admits the presentation $\langle x | - \rangle$, where $x$ is the homotopy class of the core of $T_{int}$, cf Figure \ref{fig torus x}.

\begin{figure}[!h]
\centering

\begin{tikzpicture}[every path/.style={string ,black} , every node/.style={transform shape , knot crossing , inner sep=1.5 pt } ]

\begin{scope}[xshift=0cm,rotate=0,scale=1]

\draw (-2,-0.8) node {$T_{int}$} ;

\coordinate (ch) at (-2,-1) ;
\coordinate (cb) at (-2,-2) ;
\draw [color=white,fill=gray!25] (ch) ..controls +(0.5,0) and +(0.5,0).. (cb) ..controls +(-0.5,0) and +(-0.5,0).. (ch) ;
\draw [style=dotted] (ch) ..controls +(0.5,0) and +(0.5,0).. (cb) ;
\draw (ch) ..controls +(-0.5,0) and +(-0.5,0).. (cb) ;

\coordinate (t1) at (-3,2) ;
\coordinate (t2) at (3,2) ;
\coordinate (t3) at (-3,1) ;
\coordinate (t4) at (3,1) ;
\coordinate (t5) at (-3,-1) ;
\coordinate (t6) at (3,-1) ;
\coordinate (t7) at (-3,-2) ;
\coordinate (t8) at (3,-2) ;
\draw (t1) -- (t2) ..controls +(2.5,0) and +(2.5,0).. (t8) -- (t7) ..controls +(-2.5,0) and +(-2.5,0).. (t1) ;
\draw (t3) -- (t4) ..controls +(1,0) and +(1,0).. (t6) -- (t5) ..controls +(-1,0) and +(-1,0).. (t3) ;

\coordinate (k1) at (-3,1.8) ;
\coordinate (k2) at (3,1.8) ;
\coordinate (k3) at (-3,1.2) ;
\coordinate (k4) at (3,1.2) ;
\coordinate (k5) at (-3,-1.2) ;
\coordinate (k6) at (3,-1.2) ;
\coordinate (k7) at (-3,-1.8) ;
\coordinate (k8) at (3,-1.8) ;
\coordinate (k9) at (-3,1.5) ;
\coordinate (k10) at (3,1.5) ;
\coordinate (k11) at (3,-1.5) ;
\coordinate (k12) at (-3,-1.5) ;
\draw (k2) ..controls +(2.3,0) and +(2.3,0).. (k8) -- (k7) ..controls +(-2.3,0) and +(-2.3,0).. (k1) ;
\draw (k4) ..controls +(1.3,0) and +(1.3,0).. (k6) -- (k5) ..controls +(-1.3,0) and +(-1.3,0).. (k3) ;
\draw [style=dotted] (k10) ..controls +(1.8,0) and +(1.8,0).. (k11) -- (k12) ..controls +(-1.8,0) and +(-1.8,0).. (k9) ;

\begin{scope}[xshift=-3cm]
\coordinate (kh) at (0,1.8) ;
\coordinate (km) at (0,1.5) ;
\coordinate (kb) at (0,1.2) ;
\coordinate (khi) at (0.5,2) ;
\coordinate (kbi) at (1,1) ;
\coordinate (kh') at (1.5,1.8) ;
\coordinate (km') at (1.5,1.5) ;
\coordinate (kb') at (1.5,1.2) ;
\draw  (kh) ..controls +(0.1,0) and +(-0.1,0).. (kb') ;
\draw [style=dotted] (km) ..controls +(0.1,0) and +(-0.1,0).. (khi) ;
\draw  (khi) ..controls +(0.1,0) and +(-0.1,0).. (kh') ;
\draw  (kb) ..controls +(0.1,0) and +(-0.1,0).. (kbi) ;
\draw [style=dotted] (kbi) ..controls +(0.1,0) and +(-0.1,0).. (km') ;
\end{scope}

\begin{scope}[xshift=-1.5cm]
\coordinate (kh) at (0,1.8) ;
\coordinate (km) at (0,1.5) ;
\coordinate (kb) at (0,1.2) ;
\coordinate (khi) at (0.5,2) ;
\coordinate (kbi) at (1,1) ;
\coordinate (kh') at (1.5,1.8) ;
\coordinate (km') at (1.5,1.5) ;
\coordinate (kb') at (1.5,1.2) ;
\draw  (kh) ..controls +(0.1,0) and +(-0.1,0).. (kb') ;
\draw [style=dotted] (km) ..controls +(0.1,0) and +(-0.1,0).. (khi) ;
\draw  (khi) ..controls +(0.1,0) and +(-0.1,0).. (kh') ;
\draw  (kb) ..controls +(0.1,0) and +(-0.1,0).. (kbi) ;
\draw [style=dotted] (kbi) ..controls +(0.1,0) and +(-0.1,0).. (km') ;
\end{scope}

\begin{scope}[xshift=0cm]
\coordinate (kh) at (0,1.8) ;
\coordinate (km) at (0,1.5) ;
\coordinate (kb) at (0,1.2) ;
\coordinate (khi) at (0.5,2) ;
\coordinate (kbi) at (1,1) ;
\coordinate (kh') at (1.5,1.8) ;
\coordinate (km') at (1.5,1.5) ;
\coordinate (kb') at (1.5,1.2) ;
\draw  (kh) ..controls +(0.1,0) and +(-0.1,0).. (kb') ;
\draw [style=dotted] (km) ..controls +(0.1,0) and +(-0.1,0).. (khi) ;
\draw  (khi) ..controls +(0.1,0) and +(-0.1,0).. (kh') ;
\draw  (kb) ..controls +(0.1,0) and +(-0.1,0).. (kbi) ;
\draw [style=dotted] (kbi) ..controls +(0.1,0) and +(-0.1,0).. (km') ;
\end{scope}

\begin{scope}[xshift=1.5cm]
\coordinate (kh) at (0,1.8) ;
\coordinate (km) at (0,1.5) ;
\coordinate (kb) at (0,1.2) ;
\coordinate (khi) at (0.5,2) ;
\coordinate (kbi) at (1,1) ;
\coordinate (kh') at (1.5,1.8) ;
\coordinate (km') at (1.5,1.5) ;
\coordinate (kb') at (1.5,1.2) ;
\draw  (kh) ..controls +(0.1,0) and +(-0.1,0).. (kb') ;
\draw [style=dotted] (km) ..controls +(0.1,0) and +(-0.1,0).. (khi) ;
\draw  (khi) ..controls +(0.1,0) and +(-0.1,0).. (kh') ;
\draw  (kb) ..controls +(0.1,0) and +(-0.1,0).. (kbi) ;
\draw [style=dotted] (kbi) ..controls +(0.1,0) and +(-0.1,0).. (km') ;
\end{scope}

\coordinate (x9) at (-3,1.4) ;
\coordinate (x10) at (3,1.4) ;
\coordinate (x11) at (3,-1.4) ;
\coordinate (x12) at (-3,-1.4) ;
\coordinate (x13) at (0,-1.4) ;

\draw [->] [color=blue] (x10) ..controls +(1.7,0) and +(1.7,0).. (x11) -- (x13);
\draw [color=blue] (x13) -- (x12) ..controls +(-1.7,0) and +(-1.7,0).. (x9)--(x10) ;

\draw [color=blue] (0,-1.7) node {$x$} ;

\end{scope}
\end{tikzpicture}

\caption{The generator $x$, core of $T_{int}$} \label{fig torus x}
\end{figure}
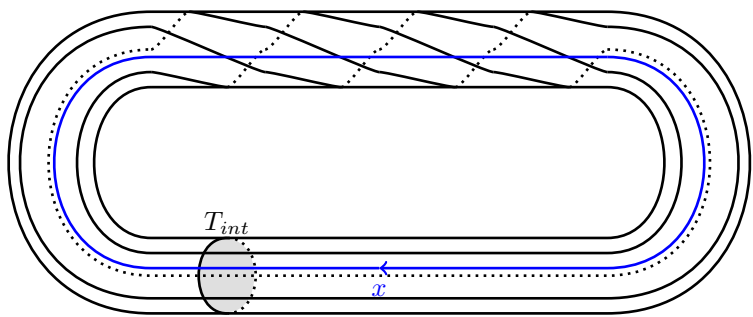

$V$ is homeomorphic to an annulus, thus $G_0$ admits the presentation $\langle z| - \rangle $ where the generator $z$ is drawn on Figure \ref{fig torus z}. Note that $z$ follows the direction of the strands, that is the same as the one of the core if $p>0$ and the opposite if $p<0$.

\begin{figure}[!h]
\centering
\begin{tikzpicture}[every path/.style={string } , every node/.style={transform shape , knot crossing , inner sep=1.5 pt } ]

\begin{scope}[xshift=0cm,rotate=0,scale=1]

\draw (-2,-0.8) node {$T_{int}$} ;

\coordinate (ch) at (-2,-1) ;
\coordinate (cb) at (-2,-2) ;
\draw [style=dotted] (ch) ..controls +(0.5,0) and +(0.5,0).. (cb) ;
\draw (ch) ..controls +(-0.5,0) and +(-0.5,0).. (cb) ;

\coordinate (t1) at (-3,2) ;
\coordinate (t2) at (3,2) ;
\coordinate (t3) at (-3,1) ;
\coordinate (t4) at (3,1) ;
\coordinate (t5) at (-3,-1) ;
\coordinate (t6) at (3,-1) ;
\coordinate (t7) at (-3,-2) ;
\coordinate (t8) at (3,-2) ;
\draw (t1) -- (t2) ..controls +(2.5,0) and +(2.5,0).. (t8) -- (t7) ..controls +(-2.5,0) and +(-2.5,0).. (t1) ;
\draw (t3) -- (t4) ..controls +(1,0) and +(1,0).. (t6) -- (t5) ..controls +(-1,0) and +(-1,0).. (t3) ;

\coordinate (k1) at (-3,1.8) ;
\coordinate (k2) at (3,1.8) ;
\coordinate (k3) at (-3,1.2) ;
\coordinate (k4) at (3,1.2) ;
\coordinate (k5) at (-3,-1.2) ;
\coordinate (k6) at (3,-1.2) ;
\coordinate (k7) at (-3,-1.8) ;
\coordinate (k8) at (3,-1.8) ;
\coordinate (k9) at (-3,1.5) ;
\coordinate (k10) at (3,1.5) ;
\coordinate (k11) at (3,-1.5) ;
\coordinate (k12) at (-3,-1.5) ;
\draw (k2) ..controls +(2.3,0) and +(2.3,0).. (k8) -- (k7) ..controls +(-2.3,0) and +(-2.3,0).. (k1) ;
\draw (k4) ..controls +(1.3,0) and +(1.3,0).. (k6) -- (k5) ..controls +(-1.3,0) and +(-1.3,0).. (k3) ;
\draw [style=dotted] (k10) ..controls +(1.8,0) and +(1.8,0).. (k11) -- (k12) ..controls +(-1.8,0) and +(-1.8,0).. (k9) ;

\begin{scope}[xshift=-3cm]
\coordinate (kh) at (0,1.8) ;
\coordinate (km) at (0,1.5) ;
\coordinate (kb) at (0,1.2) ;
\coordinate (khi) at (0.5,2) ;
\coordinate (kbi) at (1,1) ;
\coordinate (kh') at (1.5,1.8) ;
\coordinate (km') at (1.5,1.5) ;
\coordinate (kb') at (1.5,1.2) ;
\draw  (kh) ..controls +(0.1,0) and +(-0.1,0).. (kb') ;
\draw [style=dotted] (km) ..controls +(0.1,0) and +(-0.1,0).. (khi) ;
\draw  (khi) ..controls +(0.1,0) and +(-0.1,0).. (kh') ;
\draw  (kb) ..controls +(0.1,0) and +(-0.1,0).. (kbi) ;
\draw [style=dotted] (kbi) ..controls +(0.1,0) and +(-0.1,0).. (km') ;
\coordinate (kbr) at (0.25,1) ;
\coordinate (khr) at (1.25,2) ;
\draw [style=dotted, color=red] (kbr) ..controls +(0.1,0) and +(-0.1,0).. (khr) ;
\draw [color=red] (0,1.1) ..controls +(0.1,0) and +(-0.1,0).. (kbr) ;
\draw [->] [color=red] (khr) ..controls +(0.1,0) and +(-0.1,0).. (1.5,1.9) ;
\draw [->] [color=red] (0,1.9) ..controls +(0.3,0) and +(-0.1,0).. (1.5,1.4) ;
\draw [->] [color=red] (0,1.4) ..controls +(0.1,0) and +(-0.1,0).. (1.5,1.1) ;
\end{scope}

\begin{scope}[xshift=-1.5cm]
\coordinate (kh) at (0,1.8) ;
\coordinate (km) at (0,1.5) ;
\coordinate (kb) at (0,1.2) ;
\coordinate (khi) at (0.5,2) ;
\coordinate (kbi) at (1,1) ;
\coordinate (kh') at (1.5,1.8) ;
\coordinate (km') at (1.5,1.5) ;
\coordinate (kb') at (1.5,1.2) ;
\draw  (kh) ..controls +(0.1,0) and +(-0.1,0).. (kb') ;
\draw [style=dotted] (km) ..controls +(0.1,0) and +(-0.1,0).. (khi) ;
\draw  (khi) ..controls +(0.1,0) and +(-0.1,0).. (kh') ;
\draw  (kb) ..controls +(0.1,0) and +(-0.1,0).. (kbi) ;
\draw [style=dotted] (kbi) ..controls +(0.1,0) and +(-0.1,0).. (km') ;
\coordinate (kbr) at (0.25,1) ;
\coordinate (khr) at (1.25,2) ;
\draw [style=dotted, color=red] (kbr) ..controls +(0.1,0) and +(-0.1,0).. (khr) ;
\draw [color=red] (0,1.1) ..controls +(0.1,0) and +(-0.1,0).. (kbr) ;
\draw [->] [color=red] (khr) ..controls +(0.1,0) and +(-0.1,0).. (1.5,1.9) ;
\draw [->] [color=red] (0,1.9) ..controls +(0.3,0) and +(-0.1,0).. (1.5,1.4) ;
\draw [->] [color=red] (0,1.4) ..controls +(0.1,0) and +(-0.1,0).. (1.5,1.1) ;
\end{scope}

\begin{scope}[xshift=0cm]
\coordinate (kh) at (0,1.8) ;
\coordinate (km) at (0,1.5) ;
\coordinate (kb) at (0,1.2) ;
\coordinate (khi) at (0.5,2) ;
\coordinate (kbi) at (1,1) ;
\coordinate (kh') at (1.5,1.8) ;
\coordinate (km') at (1.5,1.5) ;
\coordinate (kb') at (1.5,1.2) ;
\draw  (kh) ..controls +(0.1,0) and +(-0.1,0).. (kb') ;
\draw [style=dotted] (km) ..controls +(0.1,0) and +(-0.1,0).. (khi) ;
\draw  (khi) ..controls +(0.1,0) and +(-0.1,0).. (kh') ;
\draw  (kb) ..controls +(0.1,0) and +(-0.1,0).. (kbi) ;
\draw [style=dotted] (kbi) ..controls +(0.1,0) and +(-0.1,0).. (km') ;
\coordinate (kbr) at (0.25,1) ;
\coordinate (khr) at (1.25,2) ;
\draw [style=dotted, color=red] (kbr) ..controls +(0.1,0) and +(-0.1,0).. (khr) ;
\draw [color=red] (0,1.1) ..controls +(0.1,0) and +(-0.1,0).. (kbr) ;
\draw [->] [color=red] (khr) ..controls +(0.1,0) and +(-0.1,0).. (1.5,1.9) ;
\draw [->] [color=red] (0,1.9) ..controls +(0.3,0) and +(-0.1,0).. (1.5,1.4) ;
\draw [->] [color=red] (0,1.4) ..controls +(0.1,0) and +(-0.1,0).. (1.5,1.1) ;
\end{scope}

\begin{scope}[xshift=1.5cm]
\coordinate (kh) at (0,1.8) ;
\coordinate (km) at (0,1.5) ;
\coordinate (kb) at (0,1.2) ;
\coordinate (khi) at (0.5,2) ;
\coordinate (kbi) at (1,1) ;
\coordinate (kh') at (1.5,1.8) ;
\coordinate (km') at (1.5,1.5) ;
\coordinate (kb') at (1.5,1.2) ;
\draw  (kh) ..controls +(0.1,0) and +(-0.1,0).. (kb') ;
\draw [style=dotted] (km) ..controls +(0.1,0) and +(-0.1,0).. (khi) ;
\draw  (khi) ..controls +(0.1,0) and +(-0.1,0).. (kh') ;
\draw  (kb) ..controls +(0.1,0) and +(-0.1,0).. (kbi) ;
\draw [style=dotted] (kbi) ..controls +(0.1,0) and +(-0.1,0).. (km') ;
\coordinate (kbr) at (0.25,1) ;
\coordinate (khr) at (1.25,2) ;
\draw [style=dotted, color=red] (kbr) ..controls +(0.1,0) and +(-0.1,0).. (khr) ;
\draw [color=red] (0,1.1) ..controls +(0.1,0) and +(-0.1,0).. (kbr) ;
\draw [->] [color=red] (khr) ..controls +(0.1,0) and +(-0.1,0).. (1.5,1.9) ;
\draw [->] [color=red] (0,1.9) ..controls +(0.3,0) and +(-0.1,0).. (1.5,1.4) ;
\draw [->] [color=red] (0,1.4) ..controls +(0.1,0) and +(-0.1,0).. (1.5,1.1) ;
\end{scope}

\coordinate (k1') at (-3,1.9) ;
\coordinate (k2') at (3,1.9) ;
\coordinate (k3') at (-3,1.1) ;
\coordinate (k4') at (3,1.1) ;
\coordinate (k5') at (-3,-1.1) ;
\coordinate (k6') at (3,-1.1) ;
\coordinate (k7') at (-3,-1.9) ;
\coordinate (k8') at (3,-1.9) ;
\coordinate (k9') at (-3,1.4) ;
\coordinate (k10') at (3,1.4) ;
\coordinate (k11') at (3,-1.4) ;
\coordinate (k12') at (-3,-1.4) ;
\draw [->] [color=red] (k2') ..controls +(2.4,0) and +(2.4,0).. (k8') -- (k7') ..controls +(-2.4,0) and +(-2.4,0).. (k1') ;
\draw [->] [color=red] (k4') ..controls +(1.2,0) and +(1.2,0).. (k6') -- (k5') ..controls +(-1.2,0) and +(-1.2,0).. (k3') ;
\draw [->] [color=red] (k10') ..controls +(1.7,0) and +(1.7,0).. (k11') -- (k12') ..controls +(-1.7,0) and +(-1.7,0).. (k9') ;

\draw [color=red] (0,-1.31) node {$z$} ;
\end{scope}
\end{tikzpicture}
\caption{The generator $z$ of $G_0$} \label{fig torus z}
\end{figure}
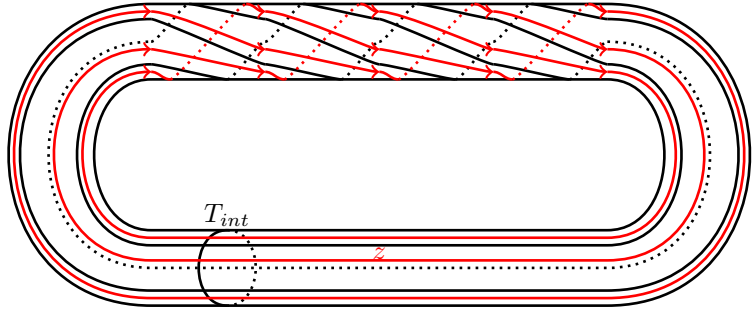

The inclusions $ V \subset U_1 $  and $V \subset U_2$ induce homotopy maps that send $z$ to $x^p$ and $y^q \lambda^p$ respectively. We hope the figures make this point clearer.

Thus, by the Seifert-van Kampen theorem, $G= \tilde{G}_{p,q}$ admits the presentation $P_{p,q} = \langle x , y , \lambda | x^p = \lambda^p y^q, \lambda y = y \lambda \rangle$.

\end{proof}

\subsection{A meridian-longitude system in the group presentation of the pattern} \label{subsection mer long}

In this subsection we will explain how to obtain in general a group presentation for $G_{P \subset T_P} = \pi_1(T_P \setminus P)$ containing the homotopy classes of a preferred meridian-longitude pair of $T_P$ as generators. This will not help us to prove Proposition \ref{prop groupe cable}, but this illustrates that the hypotheses of Lemma \ref{lem groupe sat} are not as restrictive as we could have thought.

The method will use Wirtinger presentations, and thus is not the same as the one used in Lemma \ref{lem groupe Tpq}, but it will work for any pattern $P$.

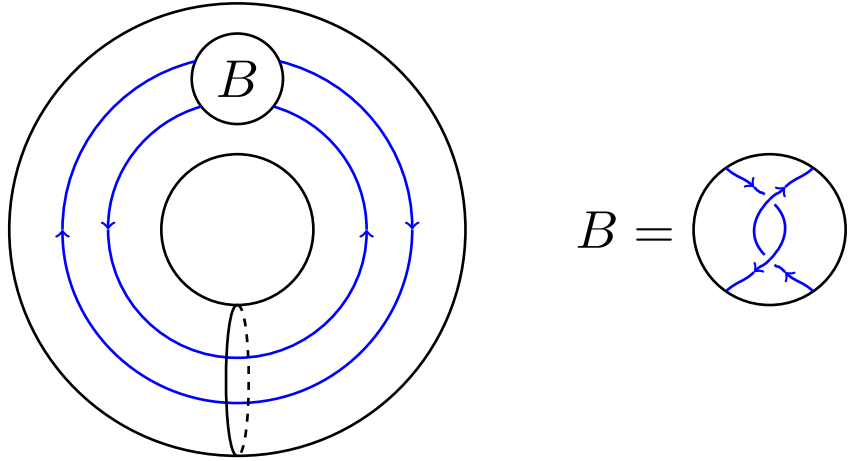
\begin{figure}[!h]
\centering
\begin{tikzpicture} [every path/.style={string } , every node/.style={transform shape , knot crossing , inner sep=1.5 pt } ] 
\begin{scope}[xshift=0cm,scale=1]
	\coordinate (n) at (-1,1.2) ;
	\coordinate (m) at (-2.6,1.5) ;
	\coordinate (O) at (0,0) ;
	\node (A) at (0.75,1.3) {};
	\node (B) at (1.25,2.17) {};
	
	\node (b) at (0,1.7) {};
	\node (b1) at (74:1.7) {};
	\node (b2) at (0:1.7) {};
	\node (b3) at (-180:1.7) {};
	\node (b4) at (-254:1.7) {};

	\node (h) at (0,2.3) {};
	\node (h1) at (76:2.3) {};
	\node (h2) at (0:2.3) {};
	\node (h3) at (-180:2.3) {};
	\node (h4) at (-256:2.3) {};

	\node (C) at (-0.75,1.3) {};
	\node (D) at (-1.25,2.17) {};	
	\coordinate (G) at (0,-1) ;
	\coordinate (H) at (0,-3) ;
	\coordinate (I) at (-0.15,2.1) ;
	\draw (O) circle (1) ;
	\draw (O) circle (3) ;	
	\node[scale=2] (o) at (0,2) {$B$};
	\draw (o) circle (0.6);
	
	\draw[color=blue] (b1) arc (74:0:1.7) ;
	\draw[color=blue][<-] (b2) arc (0:-180:1.7) ;
	\draw[color=blue][<-] (b3) arc (-180:-254:1.7) ;
		
	\draw[color=blue][->] (h1) arc (76:0:2.3) ;
	\draw[color=blue][->] (h2) arc (0:-180:2.3) ;
	\draw[color=blue] (h3) arc (-180:-256:2.3) ;

	\draw (G) ..controls +(-0.2,0) and +(-0.2,0).. (H) ;
	\draw [dashed] (G) ..controls +(0.2,0) and +(0.2,0).. (H) ;
\end{scope}

\begin{scope}[xshift=7cm,scale=1,color=blue]

	\node (o) at (0,0) {};
	\draw[color=black] (o) circle (1);

	\node[color=black,scale=2] (G) at (-1.9,0) {$B=$};

	\node (h) at (90:0.4) {};	
	\node (hd) at (55:1) {};	
	\node (hdm) at (70:0.6) {};	
	\node (hg) at (180-55:1) {};
	\node (hgm) at (180-70:0.6) {};

	\node (b) at (-90:0.4) {};
	\node (bd) at (-55:1) {};	
	\node (bdm) at (-70:0.6) {};	
	\node (bg) at (-180+55:1) {};
	\node (bgm) at (-180+70:0.6) {};

	\draw (hd.center) .. controls (hd.2 south west) and (hdm.2 north east) .. (hdm.center) ;	
	\draw[<-] (hdm.center) .. controls (hdm.2 south west) and (h.2 north east) .. (h.center) ;	
	\draw (b) .. controls (b.4 north west) and (h.4 south west) .. (h.center) ;	
	\draw (b) .. controls (b.2 south east) and (bdm.2 north west) .. (bdm.center) ;	
	\draw[<-] (bdm.center) .. controls (bdm.2 south east) and (bd.2 north west) .. (bd.center) ;

	\draw[->] (hg.center) .. controls (hg.2 south east) and (hgm.2 north west) .. (hgm.center) ;	
	\draw (hgm.center) .. controls (hgm.2 south east) and (h.2 north west) .. (h) ;	
	\draw (b.center) .. controls (b.4 north east) and (h.4 south east) .. (h) ;	
	\draw[->] (b.center) .. controls (b.2 south west) and (bgm.2 north east) .. (bgm.center) ;	
	\draw (bgm.center) .. controls (bgm.2 south west) and (bg.2 north east) .. (bg.center) ;	

\end{scope}
\end{tikzpicture}
\caption{The pattern seen as one $(m,m)$-tangle $B$ and $m$ {parallel} strands} \label{fig strands+tangle}
\end{figure}

First, notice that we can draw $P$ as $m$ parallel strands (not necessarily going in the same direction) and a $(m,m)$-tangle $B$. See Figure \ref{fig strands+tangle}, where we took $m=2$ and $P$ the Whitehead double pattern.

\begin{figure}[h]
\centering
\begin{tikzpicture} [every path/.style={string } , every node/.style={transform shape , knot crossing , inner sep=1.5 pt } ] 
\begin{scope}[xshift=0cm,scale=1]

	\node[scale=1.7,color=blue] (P) at (45:2.7) {$P$};
	\node[scale=2] (TP) at (-1.7,3) {$T_P$};

	\coordinate (n) at (-1,1.2) ;
	\coordinate (m) at (-2.6,1.5) ;
	\coordinate (O) at (0,0) ;
	\node (A) at (0.75,1.3) {};
	\node (B) at (1.25,2.17) {};
	
	\node (b) at (0,1.7) {};
	\node (b1) at (74:1.7) {};
	\node (b2) at (0:1.7) {};
	\node (b3) at (-180:1.7) {};
	\node (b4) at (-254:1.7) {};

	\node (h) at (0,2.3) {};
	\node (h1) at (76:2.3) {};
	\node (h2) at (0:2.3) {};
	\node (h3) at (-180:2.3) {};
	\node (h4) at (-256:2.3) {};

	\node (C) at (-0.75,1.3) {};
	\node (D) at (-1.25,2.17) {};	
	\coordinate (G) at (0,-1) ;
	\coordinate (H) at (0,-3) ;
	\coordinate (I) at (-0.15,2.1) ;
	\draw (O) circle (1) ;
	\draw (O) circle (3) ;	
	\node[scale=2] (o) at (0,2) {$B$};
	\draw (o) circle (0.6);
	
	\draw[color=blue] (b1) arc (74:0:1.7) ;
	\draw[color=blue][<-] (b2) arc (0:-180:1.7) ;
	\draw[color=blue][<-] (b3) arc (-180:-254:1.7) ;
		
	\draw[color=blue][->] (h1) arc (76:0:2.3) ;
	\draw[color=blue][->] (h2) arc (0:-180:2.3) ;
	\draw[color=blue] (h3) arc (-180:-256:2.3) ;

	\draw (G) ..controls +(-0.2,0) and +(-0.2,0).. (H) ;
	\draw [dashed] (G) ..controls +(0.2,0) and +(0.2,0).. (H) ;
\end{scope}

\node[scale=2] (eq) at (4,0) {$\Longleftrightarrow$};

\begin{scope}[xshift=7.3cm,scale=1]

	\node[scale=1.7,color=blue] (P) at (45:2.7) {$P$};
	\node[scale=1.7] (MP) at (0,-0.2) {$M_P$};	
	
	\node (A) at (0.75,1.3) {};
	\node (B) at (1.25,2.17) {};
	
	\node (b) at (0,1.7) {};
	\node (b1) at (74:1.7) {};
	\node (b2) at (0:1.7) {};
	\node (b3) at (-80:1.7) {};
	\node (b'3) at (-83:1.7) {};
	\node (b4) at (-100:1.7) {};
	\node (b5) at (-180:1.7) {};
	\node (b6) at (-254:1.7) {};

	\node (h) at (0,2.3) {};
	\node (h1) at (76:2.3) {};
	\node (h2) at (0:2.3) {};
	\node (h3) at (-82:2.3) {};
	\node (h'3) at (-84:2.3) {};
	\node (h4) at (-98:2.3) {};	
	\node (h5) at (-180:2.3) {};
	\node (h6) at (-256:2.3) {};

	\node (C) at (-0.75,1.3) {};
	\node (D) at (-1.25,2.17) {};	
	
	\node[scale=2] (o) at (0,2) {$B$};
	\draw (o) circle (0.6);
	
	\draw[color=blue] (b1) arc (74:0:1.7) ;
	\draw[color=blue][<-] (b2) arc (0:-77:1.7) ;
	\draw[color=blue] (b'3) arc (-83:-100:1.7) ;
	\draw[color=blue] (b4) arc (-100:-180:1.7) ;
	\draw[color=blue][<-] (b5) arc (-180:-254:1.7) ;
		
	\draw[color=blue][->] (h1) arc (76:0:2.3) ;
	\draw[color=blue] (h2) arc (0:-80:2.3) ;
	\draw[color=blue] (h'3) arc (-84:-100:2.3) ;
	\draw[color=blue][->] (h4) arc (-98:-180:2.3) ;
	\draw[color=blue] (h5) arc (-180:-256:2.3) ;

	\node (mh) at (0,-0.7) {};
	\node (mb) at (0,-3) {};

	\draw (mb.center) .. controls (mb.4  east) and (h3.4 south) .. (h3.center) ;
	\draw (h3.center) .. controls (h3.4 south) and (b3.4 north) .. (b3.center) ;
	\draw[->] (b3.center) .. controls (b3.4 north) and (mh.4  east) .. (mh.center) ;
	\draw (mh.center) .. controls (mh.4  west) and (b4.4 north) .. (b4) ;
	\draw (b4) .. controls (b4.4 south) and (h4.4 north) .. (h4) ;
	\draw[->] (h4) .. controls (h4.4 south) and (mb.4  west) .. (mb.center) ;	

\end{scope}

\end{tikzpicture}
\caption{The knot $P$ inside $T_P$ is the same as the $2$-link $P \sqcup M_P$ inside $S^3$} \label{fig torus and 2-link}
\end{figure}
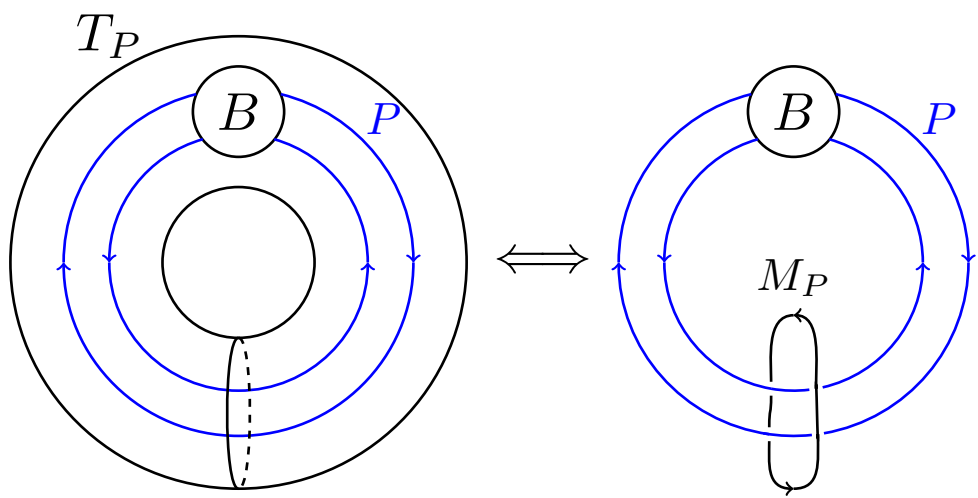

To compute a presentation of $G_{P \subset T_P} = \pi_1(T_P \setminus P)$, we remark that this group is naturally isomorphic to $G_{P \sqcup M_P} = \pi_1(S^3 \setminus (P\sqcup M_P))$ where $M_P$ is a meridian curve of $T_P$, see Figure \ref{fig torus and 2-link}.

Now we can compute a Wirtinger presentation of $G_{P \sqcup M_P}$ by the well-known process of the same name (see for example \cite[Section 3.B]{BZ}).

The Wirtinger generators are:
{\begin{itemize}
\item $\lambda$ the generator for the arc of $M_P$ that passes over the $m$ strands, which corresponds naturally to a longitude loop of $T_P$.
\item $\lambda_1, \ldots, \lambda_{m-1}$ the other generators of $M_P$, listed from the outside to the inside.
\item $a_1, \ldots, a_m$ and $a'_1, \ldots ,a'_m$ the generators for the $m$ strands of $P$, listed from the outside to the inside, such that $a'_i = \lambda a_i \lambda^{-1}$.
\item $b_1, \ldots, b_k$ the generators for the arcs strictly inside the tangle $B$.
\end{itemize}}
Figure \ref{fig wirtinger generators} pictures them partially (as always, the base point is assumed to be above the diagram).

\begin{figure}[!h]
\centering
\begin{tikzpicture} [every path/.style={string } , every node/.style={transform shape , knot crossing , inner sep=1.5 pt } ] 

\begin{scope}[xshift=0cm,scale=1]

	\node[scale=1.7,color=blue] (P) at (45:2.7) {$P$};
	\node[scale=1.7] (MP) at (0,-0.2) {$M_P$};	
	
	\node (A) at (0.75,1.3) {};
	\node (B) at (1.25,2.17) {};
	
	\node (b) at (0,1.7) {};
	\node (b1) at (74:1.7) {};
	\node (b2) at (0:1.7) {};
	\node (b3) at (-80:1.7) {};
	\node (b'3) at (-83:1.7) {};
	\node (b4) at (-100:1.7) {};
	\node (b5) at (-180:1.7) {};
	\node (b6) at (-254:1.7) {};

	\node (h) at (0,2.3) {};
	\node (h1) at (76:2.3) {};
	\node (h2) at (0:2.3) {};
	\node (h3) at (-82:2.3) {};
	\node (h'3) at (-84:2.3) {};
	\node (h4) at (-98:2.3) {};	
	\node (h5) at (-180:2.3) {};
	\node (h6) at (-256:2.3) {};

	\node (C) at (-0.75,1.3) {};
	\node (D) at (-1.25,2.17) {};	
	
	\node[scale=2] (o) at (0,2) {$B$};
	\draw (o) circle (0.6);
	
	\draw[color=blue] (b1) arc (74:0:1.7) ;
	\draw[color=blue][<-] (b2) arc (0:-77:1.7) ;
	\draw[color=blue] (b'3) arc (-83:-100:1.7) ;
	\draw[color=blue] (b4) arc (-100:-180:1.7) ;
	\draw[color=blue][<-] (b5) arc (-180:-254:1.7) ;
		
	\draw[color=blue][->] (h1) arc (76:0:2.3) ;
	\draw[color=blue] (h2) arc (0:-80:2.3) ;
	\draw[color=blue] (h'3) arc (-84:-100:2.3) ;
	\draw[color=blue][->] (h4) arc (-98:-180:2.3) ;
	\draw[color=blue] (h5) arc (-180:-256:2.3) ;

	\node (mh) at (0,-0.7) {};
	\node (mb) at (0,-3) {};

	\draw (mb.center) .. controls (mb.4  east) and (h3.4 south) .. (h3.center) ;
	\draw (h3.center) .. controls (h3.4 south) and (b3.4 north) .. (b3.center) ;
	\draw[->] (b3.center) .. controls (b3.4 north) and (mh.4  east) .. (mh.center) ;
	\draw (mh.center) .. controls (mh.4  west) and (b4.4 north) .. (b4) ;
	\draw (b4) .. controls (b4.4 south) and (h4.4 north) .. (h4) ;
	\draw[->] (h4) .. controls (h4.4 south) and (mb.4  west) .. (mb.center) ;

\end{scope}

\begin{scope}[color=magenta]

\node (aa0) at (150:1.1) {$a_2$} ;
	\node (aa1) at (150:1.3) {} ;
	\node (aa2) at (150:1.68) {} ;
	\node (aa3) at (150:1.72) {} ;
	\node (aa4) at (150:2.1) {} ;
	\draw (aa4) -- (aa3) ;
	\draw[->] (aa2) -- (aa1) ;

	\node (a1) at (170:1.9) {} ;
	\node (a2) at (170:2.28) {} ;
	\node (a3) at (170:2.32) {} ;
	\node (a4) at (170:2.7) {} ;
	\node (a0) at (170:2.9) {$a_1$} ;
	\draw[<-] (a4) -- (a3) ;
	\draw (a2) -- (a1) ;

	\node (aa'0) at (40:1.1) {$a'_2$} ;
	\node (aa'1) at (40:1.3) {} ;
	\node (aa'2) at (40:1.68) {} ;
	\node (aa'3) at (40:1.72) {} ;
	\node (aa'4) at (40:2.1) {} ;
	\draw (aa'4) -- (aa'3) ;
	\draw[->] (aa'2) -- (aa'1) ;

	\node (a'1) at (10:1.9) {} ;
	\node (a'2) at (10:2.28) {} ;
	\node (a'3) at (10:2.32) {} ;
	\node (a'4) at (10:2.7) {} ;
	\node (a'0) at (10:2.9) {$a'_1$} ;
	\draw[<-] (a'4) -- (a'3) ;
	\draw (a'2) -- (a'1) ;
	
	\node (l1) at (-0.2,-1) {} ;
	\node (l2) at (0.28,-1) {} ;
	\node (l3) at (0.3,-1) {} ;
	\node (l4) at (0.7,-1) {} ;
	\draw (l4) -- (l3) ;
	\draw[->] (l2) -- (l1) ;
	\node (l0) at (0,-1.25) {$\lambda$} ;

	\node (ll1) at (0.2,-1.85) {} ;
	\node (ll2) at (-0.29,-1.85) {} ;
	\node (ll3) at (-0.3,-1.83) {} ;
	\node (ll4) at (-0.7,-1.83) {} ;
	\draw (ll4) -- (ll3) ;
	\draw[->] (ll2) -- (ll1) ;
	\node (ll0) at (0,-2.1) {$\lambda_1$} ;	
	
\end{scope}

\begin{scope}[color=red]

	\node (m0) at (220:3.2) {$\mu$} ;
	\node (m1) at (220:1.2) {} ;
	\node (m2) at (210:1.7) {} ;
	\node (m3) at (213:2.3) {} ;
	\node (m4) at (220:2.9) {} ;

	\draw (m1.center) .. controls (m1.2  north west) and (m2.2 east) .. (m2) ;
	\draw (m2) .. controls (m2.2 south west) and (m3.2 north east) .. (m3) ;
	\draw (m3) .. controls (m3.2  west) and (m4.4 north west) .. (m4.center) ;
	\draw[->] (m4.center) .. controls (m4.4  south east) and (m1.8 south east) .. (m1.center) ;

\end{scope}

\end{tikzpicture}

\caption{The Wirtinger generators} \label{fig wirtinger generators}
\end{figure}
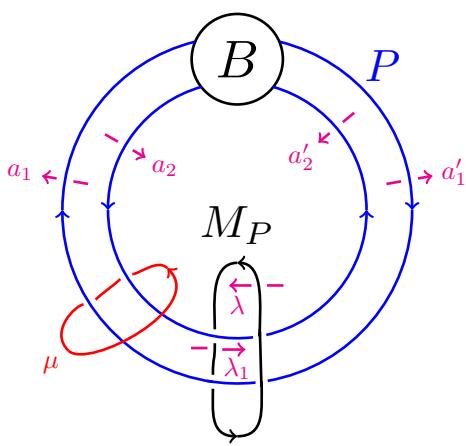

Note that we can assume that the $a_i$ and the $a'_i$ are all distinct, since we can add a first Reidemeister move twist at each of the $2m$ points of entrance of $P$ into $B$. 

The relators are: 
{\begin{itemize}
\item $r_1, \ldots , r_{m+k-1}$, some words in the $a_i$, $a'_i$ and $b_j$, corresponding to the crossings inside $B$.
\item $a'_i = \lambda a_i \lambda^{-1}$ for the crossings where $M_P$ passes over $P$.
\item $\lambda_1 = a_1^{e_1} \lambda a_1^{-e_1} ,
\lambda_2 = a_2^{e_2} \lambda_1 a_2^{-e_2},
\ldots ,
\lambda = a_m^{e_m} \lambda_{m-1} a_m^{-e_m}$ for the crossings where $M_P$ passes under $P$ (here $e_i = \pm 1$ depends on the orientation of the $i$-th strand).
\end{itemize}}

Thus $G_{P \sqcup M_P}$ admits the Wirtinger presentation
$$Q= \langle 
a_i , a'_i,  b_j, \lambda_{\alpha}, \lambda
| r_l, 
a'_i = \lambda a_i \lambda^{-1},
\lambda_1 = a_1^{e_1} \lambda a_1^{-e_1} ,
\ldots ,
\lambda = a_m^{e_m} \lambda_{m-1} a_m^{-e_m}
  \rangle,$$
where $i = 1, \ldots , m$, $j= 1 , \ldots k$, $ \alpha = 1, \ldots , m-1$ and  $l=1, \ldots, m+k-1$.  
 
A preferred longitude of $T_P$ is among the generators of $Q$, as $\lambda$. We also want a meridian loop $\mu$. As shown in Figure \ref{fig wirtinger generators}, $\mu$ is equal to $a_m^{e_m} \ldots a_1^{e_1}$.
We can thus write
$$Q_1= \langle 
a_i , a'_i,  b_j, \lambda_{\alpha}, \lambda, \mu
| r_l, 
a'_i = \lambda a_i \lambda^{-1},
\lambda_1 = a_1^{e_1} \lambda a_1^{-e_1} ,
\ldots ,
\lambda = a_m^{e_m} \lambda_{m-1} a_m^{-e_m}, 
\mu = a_m^{e_m} \ldots a_1^{e_1}
  \rangle$$
an other presentation of $G_{P \sqcup M_P}$, that has the form we wanted.

Now we can simplify this presentation and get rid of the generators $\lambda_{\alpha}$.

By substituting $\lambda_{\alpha}$ with $a_{\alpha}^{e_{\alpha}} \lambda_{\alpha-1} a_\alpha^{-e_\alpha}$ from $\alpha = 1$ to $m-1$ (with the convention $\lambda_0 = \lambda$), we obtain the simplified presentation 
$$Q_2= \langle 
a_i , a'_i,  b_j, \lambda, \mu
| r_l, 
a'_i = \lambda a_i \lambda^{-1},
\lambda = (a_m^{e_m} \ldots a_1^{e_1}) \lambda  (a_1^{-e_1} \ldots a_m^{-e_m}) ,
\mu = a_m^{e_m} \ldots a_1^{e_1}
  \rangle$$
that is equivalent to 
$$Q_3= \langle 
a_i , a'_i,  b_j, \lambda, \mu
| r_l, 
a'_i = \lambda a_i \lambda^{-1},
\lambda \mu =  \mu\lambda,
\mu = a_m^{e_m} \ldots a_1^{e_1}
  \rangle.$$
  
In conclusion, the group of the pattern knot $P$ inside its solid torus $T_P$ admits a group presentation of the form of $Q_3$. This presentation is simple in the sense that the generators $a_i, a'_i, b_j$ and the relators $r_l$ can all be read of the diagram of $P$. Moreover, $Q_3$ contains a preferred meridian-longitude pair of $T_P$ in its generators.

\begin{rem}
This method gives us the (simplified) presentation
$$ \langle b, \lambda, \mu | \lambda \mu \lambda^{-1} \mu^{-1}, 
b \lambda b \lambda^{-1} b^{-1} \lambda \mu b^{-1} \lambda^{-1} \rangle$$
for the Whitehead link.
\end{rem}

\subsection{Group presentation of a satellite knot}

The following lemma gives us a group presentation of the satellite knot when we know a presentation of the pattern group with a preferred meridian-longitude pair of the pattern torus among its generators and any presentation of the companion group.

\begin{lem} \label{lem groupe sat}
Let $T$ be a tubular neighboorhood of $T_C$ distinct from it. We will take $pt$ any point in $T \smallsetminus T_C$, it will be the basepoint for all the following fundamental groups.
Notice that  $G_{P \subset T_P} =  \pi_1(T \setminus S_{C,P})$ is isomorphic to ${\pi_1(T_P\setminus P,~pt')}$ where $pt' = h_{PC}^{-1}(pt)$.

Suppose there exists $P_{P \subset T_P} = \langle b_1 , \ldots , b_{l-1}, \lambda, \mu | s_1 , \ldots, s_{l} \rangle $ a presentation of $G_{P \subset T_P}$ where $\lambda$ and $\mu$ are the homotopy classes of a  longitude curve and a meridian curve of $T_P$.

Then there exists a presentation $P_C = \langle a_1 , \ldots , a_k | r_1 , \ldots, r_{k-1}   \rangle$ of $G_C$
and a presentation
$$P_S = \langle a_1 , \ldots , a_k, b_1 , \ldots , b_{l-1}, \lambda, \mu | r_1 , \ldots, r_{k-1}, s_1, \ldots, s_{l-1}, 
\lambda^{-1} W(a_i) , a_k^{-1} \mu
   \rangle$$
of $G_S = \pi_1(S^3 \setminus S_{C,P})$, with $W(a_i)$ a word in the $a_i, i= 1 , \ldots ,k$.
\end{lem}

\begin{proof}

We will use the Seifert-van Kampen theorem with the basepoint $pt$. We denote $W= S^3 \smallsetminus S_{C,P}$, $U_C = S^3 \smallsetminus \overline{T_C}$, $U_P= T \smallsetminus S_{C,P}$, $V = T \smallsetminus \overline{T_C}$, and $G_S , G_C , G_{P \subset T_P} , G_0$ their respective fundamental groups.

The drawings of Figure \ref{fig BIG} are meant to represent an angular fraction of the $C$-shaped sets, a fraction that contains the  \guillemotleft  essence of the pattern $P$\guillemotright \ and also the basepoint $pt$. They are here to make perfectly clear what $W, U_C, U_P, V$ are.

\begin{figure}[!h]
\centering

\begin{tikzpicture}[every path/.style={string ,black} , every node/.style={transform shape , knot crossing , inner sep=1.5 pt } ]

\begin{scope}[xshift=-2cm,yshift=0cm,scale=0.5]

\coordinate (nom) at (90:12.5) ;
\draw[scale=2] (nom) node {$W$} ;
\coordinate (tc) at (80:14.6) ;
\draw[scale=2] (tc) node {$S_{C,P}$} ;

\coordinate (pt) at (89:15.4) ;
\draw[scale=2] (pt) node[above left]{$pt$} ;
\draw (pt) node {$\bullet$} ;
\coordinate (c) at (70:14-0.5);
\coordinate (cg) at (71:14-0.4);
\coordinate (cd) at (69:14-0.6);

\coordinate (h1) at (90:14+2) ;
\coordinate (h2) at (90:14+1) ;
\coordinate (h3) at (90:14+0.5) ;
\coordinate (b1) at (90:14-2) ;
\coordinate (b2) at (90:14-1) ;
\coordinate (b3) at (90:14-0.5) ;

\coordinate (h1') at (50:14+2) ;
\coordinate (h2') at (50:14+1) ;
\coordinate (h3') at (50:14+0.5) ;
\coordinate (b1') at (50:14-2) ;
\coordinate (b2') at (50:14-1) ;
\coordinate (b3') at (50:14-0.5) ;

\draw (b3) ..controls +(0.25,0) and +(-0.91,-0.42).. (c.center) ;
\draw (h3') ..controls +(-0.77/4,0.16) and +(0.91,0.42).. (c.center) ;
\draw (h3) ..controls +(0.25,0) and +(-0.42,0.91).. (cg) ;
\draw (b3') ..controls +(-0.77/4,0.16) and +(0.42,-0.91).. (cd) ;

\end{scope}

\begin{scope}[xshift=4cm,yshift=0cm,scale=0.5]

\coordinate (nom) at (90:12.5) ;
\draw[scale=2] (nom) node {$U_P$} ;
\coordinate (tt) at (80:15.9) ;
\draw[scale=2] (tt) node {$T$} ;
\coordinate (tc) at (80:14.6) ;
\draw[scale=2] (tc) node {$S_{C,P}$} ;

\coordinate (pt) at (89:15.2) ;
\draw[scale=2] (pt) node[above left]{$pt$} ;
\draw (pt) node {$\bullet$} ;
\coordinate (c) at (70:14-0.5);
\coordinate (cg) at (71:14-0.4);
\coordinate (cd) at (69:14-0.6);

\coordinate (h1) at (90:14+2) ;
\coordinate (h2) at (90:14+1) ;
\coordinate (h3) at (90:14+0.5) ;
\coordinate (b1) at (90:14-2) ;
\coordinate (b2) at (90:14-1) ;
\coordinate (b3) at (90:14-0.5) ;

\draw (h1) ..controls +(1,0) and +(1,0).. (b1) ;
\draw (h1) ..controls +(-1,0) and +(-1,0).. (b1) ;

\coordinate (h1') at (50:14+2) ;
\coordinate (h2') at (50:14+1) ;
\coordinate (h3') at (50:14+0.5) ;
\coordinate (b1') at (50:14-2) ;
\coordinate (b2') at (50:14-1) ;
\coordinate (b3') at (50:14-0.5) ;

\draw (h1') ..controls +(1,0) and +(1,0).. (b1') ;
\draw [style=dotted] (h1') ..controls +(-1,0) and +(-1,0).. (b1') ;

\draw (h1) ..controls +(1,0) and +(-0.77,0.64).. (h1') ;
\draw (b1) ..controls +(1,0) and +(-0.77,0.64).. (b1') ;

\draw (b3) ..controls +(0.25,0) and +(-0.91,-0.42).. (c.center) ;
\draw (h3') ..controls +(-0.77/4,0.16) and +(0.91,0.42).. (c.center) ;
\draw (h3) ..controls +(0.25,0) and +(-0.42,0.91).. (cg) ;
\draw (b3') ..controls +(-0.77/4,0.16) and +(0.42,-0.91).. (cd) ;

\end{scope}

\begin{scope}[xshift=-2cm,yshift=-4cm,scale=0.5]

\coordinate (nom) at (90:12.5) ;
\draw[scale=2] (nom) node {$U_C$} ;
\coordinate (tc) at (80:14.9) ;
\draw[scale=2] (tc) node {$T_C$} ;

\coordinate (pt) at (89:15.2) ;
\draw[scale=2] (pt) node[above left]{$pt$} ;
\draw (pt) node {$\bullet$} ;
\coordinate (c) at (70:14-0.5);
\coordinate (cg) at (71:14-0.4);
\coordinate (cd) at (69:14-0.6);

\coordinate (h1) at (90:14+2) ;
\coordinate (h2) at (90:14+1) ;
\coordinate (h3) at (90:14+0.5) ;
\coordinate (b1) at (90:14-2) ;
\coordinate (b2) at (90:14-1) ;
\coordinate (b3) at (90:14-0.5) ;

\draw (h2) ..controls +(0.5,0) and +(0.5,0).. (b2) ;
\draw (h2) ..controls +(-0.5,0) and +(-0.5,0).. (b2) ;

\coordinate (h1') at (50:14+2) ;
\coordinate (h2') at (50:14+1) ;
\coordinate (h3') at (50:14+0.5) ;
\coordinate (b1') at (50:14-2) ;
\coordinate (b2') at (50:14-1) ;
\coordinate (b3') at (50:14-0.5) ;

\draw (h2') ..controls +(0.5,0) and +(0.5,0).. (b2') ;
\draw[style=dashed] (h2') ..controls +(-0.5,0) and +(-0.5,0).. (b2') ;

\draw(h2) ..controls +(1,0) and +(-0.77,0.64).. (h2') ;
\draw(b2) ..controls +(1,0) and +(-0.77,0.64).. (b2') ;

\end{scope}

\begin{scope}[xshift=4cm,yshift=-4cm,scale=0.5]

\coordinate (nom) at (90:12.5) ;
\draw[scale=2] (nom) node {$V$} ;
\coordinate (tt) at (80:15.9) ;
\draw[scale=2] (tt) node {$T$} ;
\coordinate (tc) at (80:14.9) ;
\draw[scale=2] (tc) node {$T_C$} ;

\coordinate (pt) at (89:15.2) ;
\draw[scale=2] (pt) node[above left]{$pt$} ;
\draw (pt) node {$\bullet$} ;
\coordinate (c) at (70:14-0.5);
\coordinate (cg) at (71:14-0.4);
\coordinate (cd) at (69:14-0.6);

\coordinate (h1) at (90:14+2) ;
\coordinate (h2) at (90:14+1) ;
\coordinate (h3) at (90:14+0.5) ;
\coordinate (b1) at (90:14-2) ;
\coordinate (b2) at (90:14-1) ;
\coordinate (b3) at (90:14-0.5) ;

\draw (h1) ..controls +(1,0) and +(1,0).. (b1) ;
\draw (h1) ..controls +(-1,0) and +(-1,0).. (b1) ;
\draw (h2) ..controls +(0.5,0) and +(0.5,0).. (b2) ;
\draw (h2) ..controls +(-0.5,0) and +(-0.5,0).. (b2) ;

\coordinate (h1') at (50:14+2) ;
\coordinate (h2') at (50:14+1) ;
\coordinate (h3') at (50:14+0.5) ;
\coordinate (b1') at (50:14-2) ;
\coordinate (b2') at (50:14-1) ;
\coordinate (b3') at (50:14-0.5) ;

\draw (h1') ..controls +(1,0) and +(1,0).. (b1') ;
\draw [style=dotted] (h1') ..controls +(-1,0) and +(-1,0).. (b1') ;
\draw [style=dashed](h2') ..controls +(0.5,0) and +(0.5,0).. (b2') ;
\draw[style=dotted] (h2') ..controls +(-0.5,0) and +(-0.5,0).. (b2') ;

\draw (h1) ..controls +(1,0) and +(-0.77,0.64).. (h1') ;
\draw[style=dashed] (h2) ..controls +(1,0) and +(-0.77,0.64).. (h2') ;
\draw (b1) ..controls +(1,0) and +(-0.77,0.64).. (b1') ;
\draw[style=dashed] (b2) ..controls +(1,0) and +(-0.77,0.64).. (b2') ;

\end{scope}

\end{tikzpicture}

\caption{The four open sets for the Seifert-van Kampen theorem} \label{fig BIG}
\end{figure}
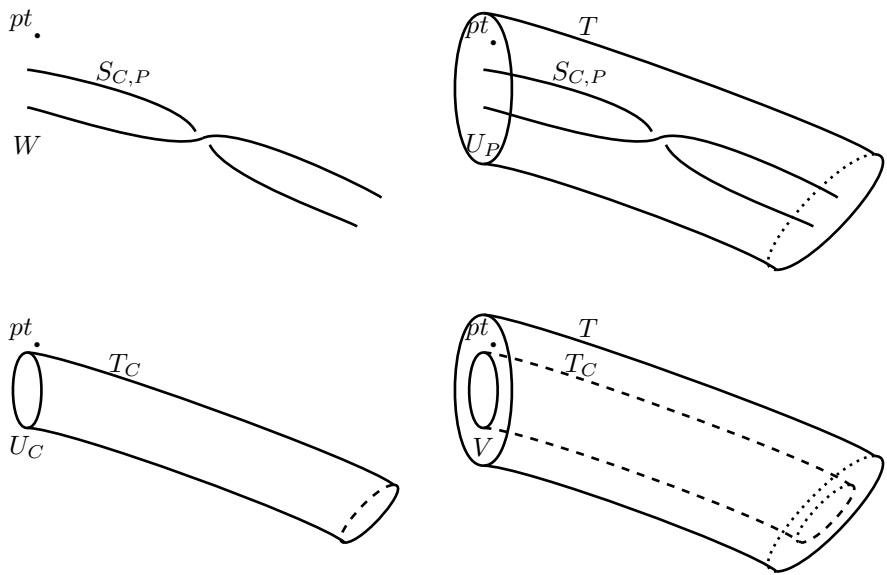

We take a Wirtinger presentation $P_C = \langle a_1 , \ldots , a_k | r_1 , \ldots, r_{k-1}   \rangle$  of \\
$G_C = \pi_1(S^3 \setminus C) = \pi_1(S^3 \setminus \overline{T_C}) = \pi_1(U_C)$ associated to a planar regular diagram projection of $C$.

We then consider $P$ inside $T_P$. The open set $U_P= T \smallsetminus S_{C,P}$ is homotopy equivalent to $T_C \smallsetminus S_{C,P}$, which is the image of $T_P \smallsetminus P$ by the homeomorphism $h_{PC}$. Thus $\pi_1(U_P) = G_{P \subset T_P}$. Let us denote $\lambda$ a longitude of $T_P$ and the corresponding element of $G_{P \subset T_P}$.

$V$ is homotopy equivalent to a $2$-torus, thus $G_0 = \langle \lambda_0 , \mu_0 | \lambda_0 \mu_0 \lambda_0^{-1} \mu_0^{-1} \rangle$, where $(\mu_0 , \lambda_0)$ is the homotopy class of a preferred meridian-longitude pair.

$V \subset U_C$ maps $\mu_0$ to any meridian loop of $G_C$, for instance $a_k$, and $\lambda_0$ to $W(a_i)$ a word in the $a_i$ such that $W(a_i)$ is a  longitude loop of the knot $C$.

$V \subset U_P$ maps $\mu_0$ to $\mu$ (a meridian loop of $\partial T_P$ that passes around the $m$ strands), and $\lambda_0$ to $\lambda$.

Hence, by the Seifert-van Kampen theorem, $$P = \langle a_1 , \ldots , a_k, b_1 , \ldots , b_{l-1}, \lambda, \mu | r_1 , \ldots, r_{k-1}, s_1, \ldots, s_{l-1}, 
\lambda^{-1} W(a_i) , a_k^{-1} \mu
   \rangle$$
is a presentation of $G_S = \pi_1(W) = \pi_1(S^3 \setminus S_{C,P})$.

\end{proof}

\subsection{Details of the proof}

Let us prove (1) of the Proposition \ref{prop groupe cable}.

Let us consider the cable knot $S$ of companion $C$ and pattern $T(p,q)$.
There exists $P_C = \langle a_1 , \ldots , a_k | r_1 , \ldots, r_{k-1}   \rangle$ a Wirtinger presentation of $G_C = \pi_1(S^3 \setminus C)$.

Lemma \ref{lem groupe Tpq} and Lemma \ref{lem groupe sat} give us the following presentation of $G_S$:
$$P = \langle a_1 , \ldots , a_k, x, y, \lambda | r_1 , \ldots, r_{k-1}, x^p y^{-q} \lambda^{-p}, y \lambda y^{-1} \lambda^{-1} ,\lambda^{-1} W(a_i) , a_k^{-1}y   \rangle$$
with $b_1$ being $x$ and $\mu$ being $y$.

Then we can suppress the relation $y \lambda = \lambda y$ because it is equivalent to $a_k W(a_i)= W(a_i) a_k$ which is already true in $G_C$ because $a_k$ is a meridian loop of the knot $C$ and $W(a_i)$ is a corresponding longitude loop. Furthermore, we can replace $y$ by $a_k$ in the relators and delete the generator $y$ and the relator $a_k^{-1} y$.

Therefore
$$P_S = \langle a_1 , \ldots , a_k, x, \lambda | r_1 , \ldots, r_{k-1}, x^p a_k^{-q} \lambda^{-p},
\lambda^{-1} W(a_i)
   \rangle$$
is a presentation of $G_S = \pi_1(S^3 \setminus S)$, with $W(a_i)$ a word in the $a_i, i= 1 , \ldots ,k$. 

Furthermore, $\lambda$ is a  longitude loop of $C$ and $x$ is the homotopy class of the core of $T_C$, since it is the image of the core of $T_P$ by $h_{PC}$.

Now let us prove  (2):

Since $\lambda$ is a  longitude loop of $C$, its linking number with $C$ is zero, thus its linking number with $S$ is zero (it is multiplied by $p$ at each crossing during the cabling process), thus $\alpha_S(\lambda) = 0$.

All the $a_i$ have the same abelianization as $a_k$, which is equal to $y$, which is a meridian loop of $\partial T_C$ and therefore circles $p$ strands. Thus $\alpha_S(y) = p$.

Finally, the relation $x^p y^{-q} \lambda^{-p}$ in $G_S$ implies that $\alpha_S(x) =q$, which concludes the proof of Proposition \ref{prop groupe cable}.

\section{Open Questions} \label{section open}

(1) The $L^2$-Alexander invariant $\Delta_K^{(2)}$ of a knot $K$ is a class of maps from a subset $X_K$ of $\R_{>0}$ to $\R_{\geqslant 0}$, up to multiplication by the $(t \mapsto t^m)$, $m \in \Z$.

We can ask many interesting questions about these maps. 

(a) Are they continuous? We know some continuity properties of the Fuglede-Kadison determinant on invertible operators, but what about the operators we use here?

(b) Are they everywhere nonzero? Or equivalently, are the operators of determinant class for all $t \in X_K$? This question can be related to the 
Determinant Conjecture (cf \cite[Chapter 13]{Luc02}).

(c) Are there knots $K$ for which $X_K$ is not the whole $\R_{>0}$? This question can be related to the strong Atiyah conjecture (cf \cite[Chapter 10]{Luc02}).

\bigskip

(2) Theorem \ref{L2 cable} gives us a cabling formula for the $L^2$-Alexander invariant. Even if an intuitive $L^2$-re-writing of the satellite formula for the Alexander polynomial stands false, cf Remark \ref{rem sat formula false}, maybe we can get a more general formula. Are there other $L^2$ satellite formulas for certain classes of pattern knots, ones that, like torus knot patterns, have group presentations that are easy to manipulate?



\end{document}